\documentclass[trsc,nonblindrev]{informs3} %
\OneAndAHalfSpacedXII 

\usepackage{natbib}
 \bibpunct[, ]{(}{)}{,}{a}{}{,}%
 \def\bibfont{\small}%
 \def\BIBand{and}%

\TheoremsNumberedThrough     

\EquationsNumberedThrough

\usepackage{mathrsfs} 
\usepackage{amsmath}
\usepackage{graphicx}

\usepackage{authblk} 
\usepackage{color}
\usepackage{titlesec}
\usepackage{verbatim}
\usepackage{indentfirst}  
\usepackage{enumerate}
\usepackage{booktabs} 
\usepackage{multirow}
\usepackage{url} 
\usepackage{array}
\usepackage{subfigure}
\usepackage{cases}
\usepackage{balance}
\usepackage[colorlinks,
    linkcolor=black,
    anchorcolor=blue,
    citecolor=blue
]{hyperref}

\TITLE{Facility Location with Congestion and Priority in Drone-Based Emergency Delivery}

\begin{document}

\ARTICLEAUTHORS{%
\AUTHOR{Xin Wang}
\AFF{Department of Industrial Engineering, Tsinghua University, Beijing 100084, China,\\
Logistics and Transportation Division, Tsinghua University Shenzhen International Graduate School, Shenzhen 518055, China,
\EMAIL{wangxin16@mails.tsinghua.edu.cn}  \URL{}}
\AUTHOR{Ruiwei Jiang}
\AFF{Department of Industrial and Operations Engineering, University of Michigan, Ann Arbor, Michigan, 48103, US \EMAIL{ruiwei@umich.edu} \URL{}}
\AUTHOR{Mingyao Qi*}
\AFF{Logistics and Transportation Division, Tsinghua University Shenzhen International Graduate School, Shenzhen 518055, China, \EMAIL{qimy@sz.tsinghua.edu.cn} \URL{}}
 } 

\ABSTRACT{%
Thanks to their fast delivery, reduced traffic restrictions, and low manpower need, drones have been increasingly deployed to deliver time-critical materials, such as medication, blood, and exam kits, in emergency situations. This paper considers a facility location model of using drones as mobile servers in emergency delivery. The model jointly optimizes the location of facilities, the capacity of drones deployed at opened facilities, and the allocation of demands, with an objective of equitable response times among all demand sites. To this end, we employ queues to model the system congestion of drone requests and consider three queuing disciplines: non-priority, static priority, and dynamic priority. For each discipline, we approximate the model as a mixed-integer second-order conic program (MISOCP), which can readily be solved in commercial solvers. We conduct extensive computational experiments to demonstrate the effectiveness and accuracy of our approach. Additionally, we compare the system performance under the three queuing disciplines and various problem parameters, from which we produce operational recommendations to decision makers in emergency delivery.
}%


\KEYWORDS{
Drone delivery; Facility location-allocation; Congestion;  Priority queue; Mixed-integer second-order cone programming.}

\maketitle

\section{Introduction}
Natural and anthropogenic disasters such as typhoons,  earthquakes, terrorist attacks, and epidemic outbreaks make severe impacts on the human kind. For example, various natural disasters have caused 449k deaths and a trillion USD economic losses between 2010 and 2019~\citep{owidnaturaldisasters}. Additionally, as of Oct. 2021, an estimated 5 million people have died from COVID-19~\citep{ourworld}. Disasters trigger high demands for emergency delivery service such as medication, blood, and exam kits.

Of all characteristics in such service, response time is arguably the most significant in deciding the survival of those affected. Unfortunately, as infrastructure damages typically arrive with the disaster, traditional ground transportation often result in delay of response, especially in remote areas. In contrast, drones, or unmanned aerial vehicles, have the advantages of fast delivery, reduced traffic restrictions, and low manpower need.
For example, it has been demonstrated that drones are 8 to 12 times faster than traditional methods in delivering testing sample tubes in Germany~\citep{qs}. In view of these advantages, many agencies have deployed drones to deliver small relief items (see, e.g.,~\cite{zipline1} and~\cite{covid19}).

A key research task for deploying drones in emergency delivery is to optimally locate the facilities, where drones load items, and distribute drones among these facilities. To this end, several families of location models have been studied in the existing literature (see, e.g.,~\cite{59.3}). In particular, covering models provide adequate coverage to all demand sites with a minimum number of facilities (see, e.g., \citet{40daskin1983maximum}, \citet{85.1marianov2002location}), median models satisfy the demand with a minimum facility and traveling cost (see, e.g., \citet{80.3berman1985optimal},  \citet{98}), and center models seek to improve the worst-case quality of service among all demand sites. Although the first two models have been widely applied, in this paper, we consider a center model in order to produce \emph{equitable} response times in emergency delivery. Here, the response time for a demand request refers to the duration from when the demand arises to when it is fulfilled. As a result, the response time consists of not only the travel time of the drone but also the (random) waiting time, which occurs if there is no available drones when the demand arises due to a congestion of drone requests. We refer to~\citet{book5.1} for existing studies on facility location with congestion, among which queuing is a natural and preferable modeling choice for the congestion at each opened facility. Unfortunately, it also yields a challenging facility location formulation. As a result, most existing studies adopt heuristic solution methods. In this paper, we derive an approximate solution method based on mixed-integer second-order conic programs (MISOCPs), which can readily be solved in commercial solvers.

In emergency situations, demand sites may not be equally important. For example, the epicenter of an earthquake usually suffers the largest impacts and would benefit the most from faster delivery of relief supplies. This motivates us to incorporate the following two priority disciplines into the queuing model.
\begin{itemize}
\item \emph{Static priority} divides the set of demand sites into different priority levels and always fulfills the demand from a higher level before doing so to the lower ones~\citep{111silva2008locating,111.3jayaswal2017facility}.
\item \emph{Dynamic priority} updates the priority of each demand site based on its inherent priority level and the waiting time it has experienced. That is, longer waiting time  promotes the priority of a demand site \citep{jackson1960some, jackson1961queues}. To our best knowledge, the dynamic priority discipline has received little attention in the facility location literature.  
\end{itemize}
As a benchmark, we also consider a queuing model with no priority levels (i.e., all demand sites have the same priority).
The main contributions of this paper are as follows: 
\begin{enumerate}
\item  We adopt three disciplines: non-priority, static priority, and dynamic priority to consider different emergency degrees and importance of heterogeneous demands  in our facility location problem with congestion and mobile servers. To our best knowledge the dynamic priority is explored for the first time.
\item We explicitly consider the system congestion via a queuing model in a facility location problem, with an objective of equitable response times among all demand sites. We employ a generally distributed service time for drones as mobile servers, that depends on the demand assignment and  facility capacity decisions.

\item We reformulate the proposed mixed-integer nonlinear programs into  MISOCPs, which can be solved directly by the existing  solver like Gurobi. Computational experiments are conducted for sensitivity analysis and comparison of the system performance under different priority disciplines by a hybrid approach of optimization and simulation. 
\end{enumerate}

The remainder of this study is organized as follows. Section~\ref{section2} reviews related literature. In Sections~\ref{section3}--\ref{section5}, we consider three variants of the proposed facility location model with non-priority, static priority, and dynamic priority, respectively. In each variant, we approximate the ensuing nonlinear and non-convex model as a MISOCP. Section~\ref{section6} reports the numerical results and  Section~\ref{section7} concludes the paper with a summary and directions for future research.

 \section{Literature Review}\label{section2}
The facility location problem for emergency service has been considered in three streams of literature, namely coverage models, median models, and center models.

Coverage models seek to maximize coverage with a given number of facilities, or alternatively, cover all demands with a minimum number of facilities. In an emergency service context, probabilistic models have been considered. 
\citet{40daskin1983maximum} considered that more than one facility may be available when a demand request arises and proposed to maximize the probability of coverage. 
Alternatively, \citet{40.3revelle1988reliability} modeled the coverage probability as a constraint (i.e., chance constraint) to meet a required level of reliability. 
These models were later extended by \citet{40.1batta1989maximal}, 
\citet{40.2repede1994developing}, 
and~\citet{85.01revelle1989maximum} 
to tackle other facility location problems. 
These models assume that the random business/idleness of all facilities are jointly independent, which may not be the case in reality. \citet{85.02marianov1994queuing} and~\citet{85marianov1996queueing} relaxed this assumption by using queues to estimate  the service busy fraction around the demand node.
Moreover, \citet{85.3marianov1998probabilistic} modeled each facility as an $M/M/k$ queue and explicitly imposed a constraint on waiting time or queue length to ensure the quality of service.
\citet{85.1marianov2002location} later extended the same model by determining the number of servers deployed at each facility instead of giving it  a priori. 


Median models locate facilities to minimize the total travel time and waiting time of customers. 
To compute customer waiting time,~\citet{80.1wang2002algorithms} and~\citet{80.11berman2006locating} modeled each opened facility as an $M/M/1$ queue. 
Since this leads to a nonlinear formulation, these studies applied heuristic algorithms. A notable attempt to provide more precise solutions is from 
\citet{100elhedhli2006service}, who recast the nonlinear formulation via piecewise linear approximations and proposed a cutting plane algorithm to find global optimum. 
Later,~\citet{98} extended this approach to $M/G/1$ queues. More recently, \citet{94.1ahmadi2017convexification} recast the same model as a MISOCP and \citet{94.4ahmadi2019linear} derived various linear reformulation of the model, as well as valid inequalities to strengthen the ensuing formulation. In addition,~\citet{94ahmadi2018location} modeled the service time as a function of the capacity of each facility and reformulated the nonlinear model as a MISOCP. 
Different from the aforementioned median models, this paper considers queues with \emph{mobile} servers (i.e., the drones) that travel to customers to deliver service. As a result, the service time depends not only on the waiting time in queues but also the server travel time, which further depends on which facility the demand is assigned to. For mobile servers, \citet{80.3berman1985optimal} studied a single facility with an $M/G/1$ queue. Other facility location models with a single mobile server include~\citet{chiu1985locating}, \citet{81.2berman1986location}, and~\citet{berman1987stochastic}, and~\citet{80.4batta1989location} extended the model to have $k$ servers at a single facility.

Center models, which this paper employs, pay more attention to an equitable quality of service among all demand nodes than the median models. In other words, our model explicitly takes into account the equity and priority among demand nodes, which is useful in an emergency delivery context. To the best of our knowledge, center models have received limited attention in existing literature. Examples of such models include \citet{85.7revelle1989maximum}, who sought to minimize the maximum service time and employed chance constraints to ensure the quality of service. More recently,~\citet{80aboolian2009multiple} modeled the congestion of facilities using an $M/M/k$ queue. They assumed the service rate of the queue to be pre-specified, while our service rates are decision-dependent.

Applications of drones in delivery service has attracted significant research attention in the recent past. 
Reviews of drone operational models can be found in  \citet{39otto2018optimization}, \citet{39.1macrina2020drone}, and \citet{39.3chung2020optimization}. The majority of the existing literature concentrates on routing problems, including traveling salesman problem with drones~\citep{8murray2015flying,22ferrandez2016optimization,17ha2018min,20agatz2018optimization,37poikonen2019branch,58salama2020joint} and vehicle routing problems with drones~\citep{24wang2017vehicle, 27poikonen2017vehicle, 28daknama2017vehicle, 57wang2019vehicle, 59kitjacharoenchai2020two}. However, location models with drones, especially for emergency service, have received much less attention by far. \citet{5scott2017drone} summarized the applications of drone delivery in healthcare industry. 
\cite{59} studied a multi-level facility location problem considering drones, recharge stations, and relief centers without congestion. In contrast,~\citet{8boutilier2019response} used an $M/M/k$ queue to model the congestion of a drone-based system for automated external defibrillators delivery. They also introduced a chance constraint to guarantee the service level. Our work is closely related to 
\citet{103.b}, who modeled the system congestion through an $M/G/k$ queue for each facility, analyzed the uncertain characteristics using the fuzzy theory, and provided a genetic algorithm. This paper is different from~\citet{103.b} in three aspects: 
(i) we adopt a center model for equitable emergency delivery;
(ii) we consider $M/G/1$ queues in the center model and derive an exact reformulation as a MISOCP;
(iii) we introduce static and dynamic priorities into the queuing discipline.

Lastly, demand priorities have been addressed in emergency logistics problems, including vehicle routing (see, e.g.,~\citet{20ghannadpour2014multi}, \citet{112.6oran2012location}, and \citet{112.1zhu2019emergency}), location routing (see, e.g.,~\citet{99.2li2018single}), and coverage models (see, e.g.,~\citet{111silva2008locating}, \citet{111.3jayaswal2017facility}, \citet{112cao2011emergency}, \citet{112.6oran2012location}, and  \citet{112.7mukhopadhyay2017prioritized}). In contrast, we consider both static and dynamic priorities in a center model.

\section{Non-Priority Model}\label{section3}
We consider a facility location model with no priorities. That is, we employ a queue to model the congestion of each server and adopt the first-come-first-serve (FCFS) discipline. Specifically, we consider a set $I$ of demand nodes and a set $J$ of candidate facility locations (to provide relief supplies via drones). We denote by $t_{ij}$ the travel time of a drone from demand node $i$ to facility $j$. We assume that $t_{ij}$ is known and deterministic for all $(i,j)$ pairs. In addition, we denote by $t_{max}$ the range of drones. That is, if the delivery time $t_{ij}$ exceeds $t_{max}$, then drones from facility $j$ do not deliver to demand node $i$.

Our model seeks to make three decisions: 
(i) $x$ are binary variables indicating the setup of  facilities, such that $x_j = 1$ if we open a facility in location $j$ and $x_j = 0$ otherwise;
(ii) $y$ are binary variables for the assignment of demand nodes, such that $y_{ij} = 1$ if demand node $i$ is assigned to facility $j$ and $y_{ij} = 0$ otherwise; and 
(iii) $k_j$ is an integer variable representing the number of drones deployed to facility $j$.

To model the system congestion, we assume that the demand at each node $i$ is generated from a Poisson process with rate $\lambda_i$. Then, the stream of demands arriving at facility $j$ is also Poisson with rate 
\begin{equation} \label{gamma}
\gamma_{j}=\sum_{i\in I}\lambda_iy_{ij}
\end{equation}
for all $j \in J$. It follows that, on average, each drone at facility $j$ spends time
\[
\dfrac{\sum_{i\in I}\lambda_i t_{ij}y_{ij}}{\sum_{i\in I}\lambda_i y_{ij}}
\]
on each delivery. As we equip each facility with $k_j$ drones, it appears appealing to model it as an $M/G/k_j$ queue, where the service time distribution is general. Unfortunately, it is not computational viable to compute the steady-state probability of the queue~\citet{tijms1981approximations} even if the decision variable $k_j$ is pre-specified, let alone incorporating such information into an optimization model as we seek to do. In this paper, we instead adopt an $M/G/1$ queue to approximate the congestion at each facility largely because of its computational practicality~\citet{book5.1,0.1}. As a result, the inter-service time, denoted by $\tau_j$, has the following first and second moments:
\begin{align}
\mathbb{E}[\tau_{j}]&=\dfrac{\sum_{i\in I}\lambda_i  t_{ij}y_{ij}}{k_j\sum_{i\in I}\lambda_i y_{ij}} \label{tau},\\
\mathbb{E}[\tau_{j}^2]&=\dfrac{\sum_{i\in I}\lambda_i t_{ij}^2y_{ij}}{k_j^2\sum_{i\in I}\lambda_i y_{ij}}.\label{tau2}
\end{align}
Then, following from the Pollaczek-Khintchine formula (see, e.g.,~\citet{book0}), the expected waiting time of the queue at facility $j$ equals
\begin{equation*}
W_{j}=\dfrac{\gamma_j \mathbb{E}[\tau_{j}^2]}{2(1-\gamma_j \mathbb{E}[\tau_{j}^2])},
\end{equation*}
which is re-written in terms of our decision variables
\begin{equation}
W_{j} = \frac{\sum_{i\in I}\lambda_i t_{ij}^2 y_{ij}}
{2k_j(k_j-\sum_{i\in I}\lambda_it_{ij} y_{ij})}. \label{0wj}
\end{equation}
We are now ready to present our model without priority as follows.
\begin{subequations}
\begin{align}
(\text{\bf NP}) \qquad \min_{x,y,k} \ & \ \max_{i \in I, j \in J}\left\{t_{ij} y_{ij} + W_j \right\} \label{0obj1} \\
\text{s.t.} \ & \ \sum_{j\in J} y_{ij} = 1 \qquad \forall i \in I, \label{0yij} \\
& \ y_{ij} \leq x_j \qquad \forall i \in I, \forall j \in J, \label{0yx} \\
& \ \sum_{j\in J}y_{ij} t_{ij} \leq t_{max} \qquad \forall i\in I, \label{0t} \\
& \ \sum_{i\in I}\lambda_it_{ij}y_{ij}\leq k_{j}\qquad \forall j\in J, \label{0rho} \\
& \ \sum_{j \in J} k_j \leq K \label{cons:dron-ub} \\
& \ x_j, y_{ij} \in \{0,1\}, \ k_j \in \mathbb{Z}_+, \qquad \forall i\in I, \forall j\in J. \label{0xx}
\end{align}
\end{subequations}
The objective function \eqref{0obj1} minimizes the longest expected response time, which is the expected waiting time plus the travel time of the drone, among all demand nodes. Constraints \eqref{0yij}--(\ref{0t}) ensure that each demand node is assigned to a facility that is open and within the range of drones. Constraints (\ref{0rho}) ensure that the queue at each facility is stable. Finally, Constraints \eqref{cons:dron-ub} cap the number of drones we can deploy at opened facilities. The parameter $K$ can be decided, for example, by pre-solving a side optimization model that seeks the minimum number of drones to maintain the stability of all queues, i.e., $\min_{x,y,k}\{\sum_{j \in J}k_j: \ \text{\eqref{0yij}--\eqref{0rho}, \eqref{0xx}}\}$. Then, we can set $K$ to be $(1 + \alpha)$ times of the minimum number of drones thus obtained, where $\alpha \geq 0$ can be specified by a planner depending on resource abundance.

The objective function of model (NP), particularly the waiting time $W_j$ defined in~\eqref{0wj}, appears to be nonlinear and nonconvex. As a consequence, (NP) cannot be solved by off-the-shelf optimization solvers (e.g., Gurobi) and raises a concern of computational intractability. In the following theorem, we show that the epigraph of $W_j$ as a function of $k_j$ and $y_{ij}$ admits a second-order conic representation. This produces a MISOCP reformulation of (NP), which can directly be computed by optimization solvers.
\begin{theorem} \label{thm:np}
The optimal value and the set of optimal solutions of model (NP) coincide with those of the following MISOCP:
\begin{subequations}
\begin{align}
\min \ & \ Z_{\text{max}} \\
\text{s.t.} \ & \ Z_{\text{max}} \geq t_{ij}y_{ij} + W_j \qquad \forall i \in I, \forall j \in J, \\
& \ \left\Vert \begin{matrix}
    \sqrt{2\lambda_1} t_{1j} \theta_{1j} \\
    \vdots \\
    \sqrt{2\lambda_I} t_{Ij} \theta_{Ij} \\
    k_j-\sum_{i\in I}\lambda_it_{ij}y_{ij}-W_j \\
    \end{matrix} \right\Vert_2\le k_j-\sum_{i\in I}\lambda_it_{ij}y_{ij}+W_j \qquad \forall j\in J, \label{0wo} \\
& \ \left\Vert  \begin{matrix}
    2y_{ij} \\
    \theta_{ij} - \beta_{ij} \\
    \end{matrix} \right\Vert_2\le \theta_{ij} + \beta_{ij} \qquad \forall i\in I,  j\in J, \label{0a} \\
& \ \left\Vert  \begin{matrix}
    2\beta_{ij} \\
    1-k_j \\
    \end{matrix} \right\Vert_2\le 1+k_j \qquad \forall i\in I,  j\in J, \label{0ak} \\
& \ \text{\eqref{0yij}--\eqref{0xx}}.\label{np:0xx}
\end{align}
\end{subequations}
\end{theorem}
\begin{proof}{Proof of Theorem~\ref{thm:np}.}
We start by rewriting the objective function~\eqref{0obj1} in the epigraphic form of $\min \{Z_{\text{max}}: \ Z_{\text{max}} \geq t_{ij}y_{ij} + W_j, \text{ for all } i \in I, j \in J\}$. Accordingly, we can rewrite the definition of $W_j$ in~\eqref{0wj} in the epigraphic form without loss of optimality:
\begin{equation*}
W_j \geq \frac{\sum_{i\in I}\lambda_i t_{ij}^2 y_{ij}}
{2k_j(k_j-\sum_{i\in I}\lambda_it_{ij} y_{ij})}.
\end{equation*}
In what follows, we assume that $k_j > 0$ and $k_j - \sum_{i\in I}\lambda_it_{ij} y_{ij} > 0$ and the result will remain valid when either of them equals zero. Then, multiplying both sides of the above inequality by $4(k_j - \sum_{i\in I}\lambda_it_{ij} y_{ij})$ yields that there exist $\theta_{ij} \geq 0$ for all $i \in I$ with
\begin{align}
4W_j\left(k_j-\sum_{i\in I}\lambda_it_{ij}y_{ij}\right) \ \geq \ & 2\sum_{i\in I}\lambda_it_{ij}^2\theta_{ij}^2, \label{0w1} \\
\theta_{ij}^2 \ \geq \ & \dfrac{y_{ij}}{k_j} \qquad \forall i \in I. \label{0theta}
\end{align}
Next, we recast~\eqref{0w1} and~\eqref{0theta} separately. For~\eqref{0w1}, it holds that
\begin{align*}
\text{\eqref{0w1}} \ \Longleftrightarrow \ & \ \left(k_j-\sum_{i\in I}\lambda_it_{ij}y_{ij}+W_j\right)^2 \geq 2\sum_{i\in I}\lambda_it_{ij}^2\theta_{ij}^2+\left(k_j-\sum_{i\in I}\lambda_it_{ij}y_{ij}-W_j\right)^2 \\
\Longleftrightarrow \ & \ \text{\eqref{0wo}}.
\end{align*}
For~\eqref{0theta}, we use the fact that $y_{ij}^4 = y_{ij}$ since $y_{ij}$ is binary, and it holds that
\begin{equation*}
\text{\eqref{0theta}} \ \Longleftrightarrow \ \exists \beta_{ij} \geq 0: \ \left\{\begin{array}{l}
    y_{ij} \leq \sqrt{\theta_{ij} \beta_{ij}} \\
    \beta_{ij} \leq \sqrt{k_j}
\end{array}
\right. \qquad \forall i \in I.
\end{equation*}
Then, following a similar procedure for recasting~\eqref{0w1}, we arrive at the representation~\eqref{0a}--\eqref{0ak} of~\eqref{0theta}. This completes the proof.\Halmos
\end{proof}

\section{Static Priority Model}	\label{section4}
In emergency delivery (e.g., after a natural disaster), the degrees of importance and urgency of supplies are highly heterogeneous. Materials that involve risks to human life, such as first-aid supplies and medications, may deserve higher priority over other materials like water and food. For example, China divides the emergency relief supplies into 3 major categories and 65 minor categories~\citep{NDRC}. 

This section extends the (NP) model to consider priorities among demands with different categories. Specifically, we categorize the demand requests into $R$ priority classes and a smaller index represents a higher priority. Note that this categorization only depends on the types of demand requests and not on time. Accordingly, we call the priority \emph{static} and shall generalize it to dynamic priority in Section~\ref{section5}. For each demand node $i \in I$, we denote by $v_{ir}$ the probability of the node requesting a class-$r$ demand with $\sum_{r \in R} v_{ir}=1$. Additionally, we allow to assign the demands from different classes to different facilities and extend the decision variables $y_{ij}$ to be $y^r_{ij}$, such that $y^r_{ij} = 1$ if we assign demands in class $r$ to facility $j$ and $y^r_{ij} = 0$ otherwise. Then, the arrival rate of the demands in class $r$ at facility $j$ becomes
\begin{equation} \label{1gamma}
\gamma_{jr}=\sum_{i\in I}\lambda_iv_{ir}y_{ij}^r
\end{equation}
for all $j \in J$. At each opened facility $j$, we consider a non-preemptive priority queue. That is, the queue dynamically maintains a list of demand requests in the order of their priorities and uses their arriving times as the tie breaker. Whenever a drone becomes available, it fulfills the top-ranked request. However, a drone does not interrupt (i.e., preempt) an ongoing delivery for a new request, even if the request has higher priority than what is being delivered. This is reasonable because otherwise drones would spend a significant portion of its time traveling for nothing. It follows that the inter-service time $\tau_{jr}$ of demands in class $r$ has the following first and second moments:
\begin{align}
E[\tau_{jr}]&=\dfrac{\sum_{i\in I}\lambda_i v_{ir} t_{ij}y_{ij}^r}{k_j\sum_{i\in I}\lambda_i v_{ir} y_{ij}^r} \label{1tau},\\
E[\tau_{jr}^2]&=\dfrac{\sum_{i\in I}\lambda_i v_{ir} t_{ij}^2y_{ij}^r}{k_j^2\sum_{i\in I}\lambda_i v_{ir}y_{ij}^r}. \label{1tau2}
\end{align}
Then, the expected waiting time $W_{jr}$ of class-$r$ demand at facility $j$ follows from existing results of $M/G/1$ queues with static, non-preemptive priority discipline (see, e.g.,~\citet[Ch. 4.4]{book0} 
\begin{align*}
W_{j1} & = \frac{\sum_{\ell=1}^R \gamma_{j\ell} \mathbb{E}[\tau^2_{j\ell}]}{2(1 - \gamma_{j1}\mathbb{E}[\tau_{j1}])}, \\
W_{jr} & = \dfrac{\sum_{\ell=1}^R \gamma_{j\ell} \mathbb{E}[\tau^2_{j\ell}]}{\left(1-\sum_{\ell=1}^r\gamma_{j\ell}\mathbb{E}[\tau_{j\ell}]\right)\left(1-\sum_{\ell=1}^{r-1}\gamma_{j\ell}\mathbb{E}[\tau_{j\ell}]\right)} \qquad \forall r \geq 2.
\end{align*}
Re-writing these formulae in terms of our decision variables yields
\begin{align}
W_{j1}
=&\frac{\sum_{\ell=1}^R\sum_{i\in I}\lambda_iv_{i\ell} t_{ij}^2 y_{ij}^{\ell}}{2k_j\left(k_j-\sum_{i\in I}\lambda_iv_{i1}t_{ij} y_{ij}^1\right)}, \label{1w1} \\
W_{jr}
=&\frac{\sum_{\ell=1}^R\sum_{i\in I}\lambda_iv_{i\ell} t_{ij}^2 y_{ij}^{\ell}}{2\left(k_j-\sum_{\ell=1}^r\sum_{i\in I}\lambda_iv_{i\ell}t_{ij} y_{ij}^{\ell}\right)\left(k_j-\sum_{\ell=1}^{r-1}\sum_{i\in I}\lambda_iv_{i\ell} t_{ij} y_{ij}^{\ell}\right)} \qquad \forall r \geq 2.\label{1w2}
\end{align}
This produces the following extension of model (NP) based on static priority:
\begin{subequations}
\begin{align}
(\text{\bf SP}) \qquad \min_{x,y,k} \ & \ \sum_{r=1}^R w_r \max_{i \in I, j \in J}\{t_{ij} y^r_{ij} + W_{jr}\} \\
\text{s.t.} \ & \ \sum_{j \in J} y^r_{ij} = 1 \qquad \forall i \in I, \forall r \in [R], \label{sp:con1} \\
& \ y_{ij}^r \leq x_j \qquad \forall i \in I, \forall j \in J, \forall r \in [R], \label{sp:con2} \\
& \ \sum_{j\in J}y^r_{ij} t_{ij} \leq t_{max} \qquad \forall i\in I, \forall r \in [R], \label{sp:con3} \\
& \ \sum_{r=1}^R \sum_{i\in I} \lambda_i t_{ij} y_{ij}\leq k_{j}\qquad \forall j\in J, \label{sp:con4} \\
& \ \sum_{j \in J} k_j \leq K \label{sp:con5} \\
& \ x_j, y^r_{ij} \in \{0,1\}, \ k_j \in \mathbb{Z}_+, \qquad \forall i\in I, \forall j\in J, \forall r \in [R]. \label{sp:con6}
\end{align}
\end{subequations}
In model (SP), a planner uses a weight $w_r$ to indicate the relative importance of the worst-case response time of class-$r$ demand requests, with $\sum_{r=1}^R w_r = 1$ and each $w_r \geq 0$. In addition, like in model (NP), the cap number $K$ of drones can be decided, e.g., by pre-solving a side model that seeks the minimum number of drones to maintain the stability of all queues.

As evidenced in the definitions~\eqref{1w1}--\eqref{1w2}, (SP) involves an even higher degree of non-linearity than (NP). Nevertheless, the following theorem shows that (SP) still admits a MISOCP representation and can be computed directly by an off-the-shelf solver.
\begin{theorem} \label{thm:sp}
The optimal value and the set of optimal solutions of model (SP) coincide with those of the following MISOCP:
\begin{subequations}
\begin{align}
\min \ & \ \sum_{r=1}^R w_r Z_r\\
\text{s.t.} \ & \ Z_r \geq t_{ij}y^r_{ij} + W_{jr} \qquad \forall i \in I, \forall j \in J, \forall r \in [R], \\
& \ \left\Vert \begin{matrix}
    \sqrt{2\lambda_1v_{11}} t_{1j} \theta^{\ell}_{1j} \\
    \vdots \\
    \sqrt{2\lambda_Iv_{IR}} t_{Ij} \theta^R_{Ij} \\
    k_j-\sum_{r=1}^R\sum_{i\in I}\lambda_iv_{i\ell}t_{ij}y^{\ell}_{ij}-W_{jr} \\
    \end{matrix} \right\Vert_2\leq k_j-\sum_{r=1}^R\sum_{i\in I}\lambda_iv_{i\ell}t_{ij}y^{\ell}_{ij}+W_{jr} \qquad \forall j\in J, \label{sp:ref-1} \\
& \ \left\Vert  \begin{matrix}
    2y^{\ell}_{ij} \\
    \theta^{\ell}_{ij} - \beta^{\ell}_{ij} \\
    \end{matrix} \right\Vert_2\leq \theta^{\ell}_{ij} + \beta^{\ell}_{ij} \qquad \forall i\in I, \forall j\in J, \forall r \in [R], \label{sp:ref-2} \\
& \ \left\Vert  \begin{matrix}
    2\beta^{\ell}_{ij} \\
    1-k_j + \sum_{\ell=1}^{r-1}\sum_{i\in I}\lambda_iv_{i\ell} t_{ij} y_{ij}^{\ell} \\
    \end{matrix} \right\Vert_2\leq 1+k_j - \sum_{\ell=1}^{r-1}\sum_{i\in I}\lambda_iv_{i\ell} t_{ij} y_{ij}^{\ell} \qquad \forall i\in I, \forall j\in J, \forall r \in [R], \label{sp:ref-3} \\
& \ \text{\eqref{sp:con1}--\eqref{sp:con6}}.
\end{align}
\end{subequations}
\end{theorem}
\begin{proof}{Proof of Theorem~\ref{thm:sp}.}
We start by rewriting the objective function of (SP) in the epigraphic form of $\min\{\sum_{r=1}^R w_r Z_r: Z_r \geq t_{ij} y^r_{ij} + W_{jr}, \text{ for all } i \in I, j \in J, r \in [R]\}$. In what follows, we represent the definition of $W_{jr}$, as a function of decision variables $k_j$ and $y^{\ell}_{ij}$ shown in~\eqref{1w2}, as second-order conic constraints. The representation of~\eqref{1w1} is similar and omitted. To this end, we rewrite~\eqref{1w2} in its epigraphic form without loss of optimality:
\[
W_{jr} \geq \frac{\sum_{\ell=1}^R\sum_{i\in I}\lambda_iv_{i\ell} t_{ij}^2 y_{ij}^{\ell}}{2\left(k_j-\sum_{\ell=1}^r\sum_{i\in I}\lambda_iv_{i\ell}t_{ij} y_{ij}^{\ell}\right)\left(k_j-\sum_{\ell=1}^{r-1}\sum_{i\in I}\lambda_iv_{i\ell} t_{ij} y_{ij}^{\ell}\right)}.
\]
This inequality holds if and only if there exist $\theta^{\ell}_{ij} \geq 0$ for all $\ell \in [R]$ and $i \in I$ such that
\begin{subequations}
\begin{align}
2W_{jr}\left(k_j-\sum_{\ell=1}^r\sum_{i\in I}\lambda_iv_{i\ell}t_{ij} y_{ij}^{\ell}\right) \geq & \sum_{\ell=1}^R\sum_{i\in I}\lambda_iv_{i\ell} t_{ij}^2 \left(\theta^{\ell}_{ij}\right)^2, \label{sp:note-1} \\
\left(\theta^{\ell}_{ij}\right)^2 \geq & \frac{y^{\ell}_{ij}}{k_j-\sum_{\ell=1}^{r-1}\sum_{i\in I}\lambda_iv_{i\ell} t_{ij} y_{ij}^{\ell}} \qquad \forall i \in I, \forall r \in [R]. \label{sp:note-2}
\end{align}
\end{subequations}
We recast~\eqref{sp:note-1} and~\eqref{sp:note-2} separately. For~\eqref{sp:note-1}, it holds that
\begin{align*}
& \ \text{\eqref{sp:note-1}} \\
\Longleftrightarrow \ & \ \left(k_j-\sum_{\ell=1}^r\sum_{i\in I}\lambda_iv_{i\ell}t_{ij} y_{ij}^{\ell} + W_{jr}\right)^2 \geq 2\sum_{\ell=1}^R\sum_{i\in I}\lambda_iv_{i\ell} t_{ij}^2 \left(\theta^{\ell}_{ij}\right)^2+\left(k_j-\sum_{\ell=1}^r\sum_{i\in I}\lambda_iv_{i\ell}t_{ij} y_{ij}^{\ell} - W_{jr}\right)^2 \\
\Longleftrightarrow \ & \ \text{\eqref{sp:ref-1}}.
\end{align*}
For~\eqref{sp:note-2}, we use the fact that $(y^{\ell}_{ij})^4 = y^{\ell}_{ij}$ since $y^{\ell}_{ij}$ is binary, and it holds that
\begin{equation*}
\text{\eqref{sp:note-2}} \ \Longleftrightarrow \ \exists \beta^{\ell}_{ij} \geq 0: \ \left\{\begin{array}{l}
    y^{\ell}_{ij} \leq \sqrt{\theta^{\ell}_{ij} \beta^{\ell}_{ij}} \\
    \beta^{\ell}_{ij} \leq \sqrt{k_j-\sum_{\ell=1}^{r-1}\sum_{i\in I}\lambda_iv_{i\ell} t_{ij} y_{ij}^{\ell}}
\end{array}
\right. \qquad \forall i \in I, \forall r \in [R].
\end{equation*}
Then, following a similar procedure for recasting~\eqref{sp:note-1}, we arrive at the representation~\eqref{sp:ref-2}--\eqref{sp:ref-3} of~\eqref{sp:note-2}. This completes the proof.\Halmos
\end{proof}

\section{Dynamic Priority Model}\label{section5}
\subsection{Preliminaries}
 In static-priority disciplines, the priority value of a demand  is a  fixed inherent quality  independent of time. However, in some practical situations the priority value of a demand point is determined not only by its class, but also  the time it has spent in the queue.

\citet{jackson1960some} first studied the dynamic non-preemptive priority discipline and assumed priority function at time $t$ of the following form:
 \begin{equation}\label{jack}
     q_r(t)=a_r+(t-T_r) \qquad t>T_r,
 \end{equation}
with geometrically distributed interarrival and service time. In this function, parameter $a_r$ denotes the initial priority value and $T_r$ is the arrival time of a class $r$ customer. The server selects the next customer with the highest instantaneous priority value of $q_r(t)$. 

\citet{priorityQ6} dealt with another form, where the priority function is given by
\begin{gather}
    q_r(t)=a_r+b_r(t-T_r) \qquad t\ge T_r,\notag\\
    0<b_1\le b_2\le...\le b_r,\label{qrt}\\
    a_1>a_2>...>a_p,\notag
\end{gather}
based on the queuing system with Poisson arrival patterns and general service time. They derived an expression for the expected waiting time of each class of customers with different priorities, and the bounds of this expression were also given. 

In our problem, we choose to use the priority function of  \citet{jackson1960some}, which is a special case of (\ref{qrt}) when 
\begin{equation*}
    b_1=b_2=...=b_r=1.
\end{equation*}
Therefore, the upper bounds of expected waiting time of  each class  can be finally computed as \citep{priorityQ6}
\begin{align}
      \overline{W_{r}} &=W_{0} + \sum_{i=1}^{r-1}\rho\rho_{i}\Delta a_{ir} \qquad \forall j \in J,\forall r\in R, 
\end{align}
where 
\begin{equation}
    \Delta a_{ir}=a_i-a_r.
\end{equation}
$\Delta a_{ir}$ indicates the initial priority-class gap between class $i$ and $r$, and $W_0$ denotes the expected waiting time in the $M/G/1$ queuing system without special priority disciplines.
$\rho$ is the total  utilization of the  queuing system, and $\rho_i$ is the utilization of class $i$ customers.

\subsection{Model}
In a real-life situation, given the occurrence of a sudden disaster, the priorities of the affected areas can be assessed by  an evaluation index system \citep{chinese99.4}, according to the estimated extent of damage, population density, economic level, and other indexes associated with these areas. Although it is important to reduce the response time of high priority areas, the demand from lower priority areas should not wait endlessly after the high priorities. Different from the static priority problem, we consider dynamic 
priority for the area priority problem instead. 

We now define the notation specific to this model. The priority class of each demand node is given. Let $I_r\in I$ denote the set of all demand nodes with priority $r$.  We  adopt (\ref{jack}) as the  priority value function with $ a_1>a_2>...>a_r$. When there is an available drone looking for a new demand to deliver, the one with the highest priority value will be selected. The fraction that the drone is busy at facility $j$ with priority  $r$ demand takes the value
\begin{align*}
    \rho_{jr}=\dfrac{\sum_{i\in I_r} \lambda_i t_{ij}y_{ij}}{k_j}\qquad \forall j \in J,\forall r\in R,
\end{align*}
satisfying $\rho_j=\sum_{r\in R}\rho_{jr}$.

The expected waiting time of a demand in an $M/G/1$ system under a FCFS discipline at server $j$ is denoted by $W_{j}$, which has been given by formula (\ref{0wj}) and further transformed into (\ref{0theta})$\sim$(\ref{0ak}).
Since there is no simple closed-form expression of the expected waiting time of  demands with all priorities, the upper bounds are used in our problem. This kind of approach is reasonable, because  our objective is to minimize the weighted maximum response time, where the real value will not exceed the value obtained by the upper bounds from our model. 
Then the expected waiting time  at facility $j$ under dynamic priority discipline is approximated as \begin{align}
      W_{j1}& = W_{j}\qquad \forall j \in J,\label{dp:w1}\\
     W_{jr} & = W_{j} + \sum_{i=1}^{r-1}\rho_j\rho_{jr}\Delta a_{ir} \qquad \forall j \in J,\forall r\ge 2 \notag
\end{align}
By substituting our decision variables, we obtain
\begin{align}
      W_{jr} & = W_{j}+\sum_{l=1}^{r-1} \Delta a_{lr} \frac{(\sum_{i\in I}\lambda_i t_{ij} y_{ij}) \ (\sum_{i\in I_l}\lambda_i t_{ij} y_{ij})}{k_j^2} \qquad \forall j \in J,\forall r\ge 2.\label{dp:w2}
\end{align}

Similar to the (NP) model, we show the  following extension of model  based on dynamic priority:
\begin{subequations}
\begin{align}
(\text{\bf DP}) \qquad \min_{x,y,k} \ & \ \sum_{r=1}^R w_r \max_{i \in I_r, j \in J}\{t_{ij} y_{ij} + W_{jr}\} \\
\text{s.t.} \ & \ \text{\eqref{0yij}--\eqref{0xx}}.
\end{align}
\end{subequations}

 The weight $w_r$  indicating the relative importance of the worst-case response time of class-$r$ demand requests, still should satisfy $\sum_{r=1}^R w_r = 1$ and each $w_r \geq 0$. In addition, like in model (NP), the cap number $K$ of drones can be decided, e.g., by pre-solving a side model that seeks the minimum number of drones to maintain the stability of all queues.

Due to the complexity of equations~\eqref{dp:w2}, (DP) involves an even higher degree of non-linearity than (NP) and (SP). Nevertheless, the following theorem shows that (DP) still admits a MISOCP representation and can be computed directly by an off-the-shelf solver.

\begin{theorem} \label{thm:dp}
	The optimal value and the set of optimal solutions of model (DP) coincide with those of the following MISOCP:
	\begin{subequations}
		\begin{align}
		\min \ & \ \sum_{r=1}^R w_r Z_r \\
		\text{s.t.} \ & \ Z_r \geq t_{ij}y_{ij} + W_{jr} \qquad \forall r \in [R], \forall i \in I_r, \forall j \in J,  \\
		& \label{dp:con0}\
			\ W_{jr}= W_{j}+\sum_{l=1}^r\Delta a_{lr} Q_{jl} \qquad \forall j \in J, \forall r \ge 2, \\
		&\label{dp:con1}\
		    \left \Vert \begin{matrix}
		    2\pi_{11jr}\\
		    \vdots\\
		    2\pi_{II_rjr}\\
		    Q_{jr}-k_j
		    \end{matrix} \right\Vert
		    \le  Q_{jr}+k_j \qquad \forall j \in J, \forall r\in [R],\\
		& \label{dp:con2}\   
			\left\Vert \begin{matrix} 2y_{ij}\\ p_{i\ell jr} - \pi_{i\ell jr} 
			\end{matrix}\right\Vert
		    \le  p_{i\ell jr} + \pi_{i\ell jr} + 2(1 - y_{\ell j})\qquad \forall i\in I, \forall \ell \in I_r, \forall j\in J,\forall r\in [R],\\
		& \label{dp:con3}\
			\left\Vert\begin{matrix} 2p_{i\ell jr}\\ k_j - \frac{1}{h_{i\ell jr}}y_{\ell j}\end{matrix} \right\Vert
		    \le k_j + \frac{1}{h_{i\ell jr}}y_{\ell j}\qquad \forall i\in I, \forall \ell \in I_r, \forall j\in J,\forall r\in [R],\\
		& \ \text{\eqref{0wo}--\eqref{np:0xx},\eqref{dp:w1}}.
		\end{align}
	\end{subequations}
\end{theorem}

\begin{proof}{Proof of Theorem~\ref{thm:dp}.}
We start by rewriting the objective function of (DP) in the epigraphic form of $\min\{\sum_{r=1}^R w_r Z_r: Z_r \geq t_{ij} y_{ij} + W_{jr}, \text{ for all } r \in [R], i \in I_r, j \in J \}$. In what follows, we represent the definition of $W_{jr}$($r\ge 2$) in formula \eqref{dp:w2}  as a function of decision variables $k_j$ and $y_{ij}$ shown in~\eqref{1w2}, as second-order conic constraints. We introduce a auxiliary variable
	\begin{equation}\label{dp:Q-term} 
	            Q_{jr}=\frac{(\sum_{i\in I}\lambda_i t_{ij} y_{ij}) \ (\sum_{i\in I_r}\lambda_it_{ij} y_{ij})}{k_j^2} \qquad \forall j \in J , \forall r\in [R].
	\end{equation}
Without loss of optimality, we rewrite this equation in its epigraphic form, and obtain
\begin{numcases}{\text{\eqref{dp:w2}}\ \Longleftrightarrow \ :}
	  \text{\eqref{dp:con0},} &\notag\\
     Q_{jr}  \geq \frac{(\sum_{i\in I}\lambda_i t_{ij} y_{ij}) \ (\sum_{i\in I_r}\lambda_it_{ij} y_{ij})}{k_j^2} \qquad \forall j \in J , \forall r\in [R],\notag &
     \end{numcases}
     
\begin{numcases}{\quad\ \Longleftrightarrow \ :}
	  \text{\eqref{dp:con0},} &\notag\\
     Q_{jr} k_j \geq \sum_{i \in I} \sum_{\ell \in I_r} \left(\frac{\lambda_i \lambda_{\ell} t_{ij} t_{\ell j} y_{ij} y_{\ell j}}{k_j}\right) \qquad \forall j \in J, \forall r\in [R]. \label{dp:Q-term}&
\end{numcases}

Next, we give the procedure to recast \eqref{dp:Q-term} as a MISOCP. Let $h_{i\ell jr} = \lambda_i \lambda_\ell t_{ij} t_{\ell j}$, the inequality \eqref{dp:Q-term} holds if and only if there exits $\pi_{i\ell jr} \geq 0$, for all $r \in [R],\ell \in I_r, i\in I$ and $ j\in J$, such that
\begin{align}
Q_{jr} k_j &\geq \sum_{i \in I} \sum_{\ell \in I_r} \pi_{i\ell jr}^2 \qquad \forall i\in I,\forall \ell\in I_r, \forall j\in J, \forall r\in [R],\label{dp:q2}\\
    \pi_{i\ell jr}^2 &\geq \frac{h_{i\ell jr} y_{ij} y_{\ell j}}{k_j} \qquad \forall i\in I,\forall \ell\in I_r, \forall j\in J, \forall r\in [R]. \label{dp:pi-term}
\end{align}
Then \eqref{dp:q2} is easy to be reformulated to SOC constraints \eqref{dp:con1}. We now prove that so is inequality~\eqref{dp:pi-term}. Indeed, if $y_{\ell j} = 1$ then \eqref{dp:pi-term} can be re-written as
$$
y_{ij} \leq \left(\frac{1}{h_{i\ell jr}}\pi_{i\ell jr}^2 k_j\right)^{1/4} \qquad \forall i\in I,\forall \ell \in I_r,\forall j\in J,\forall r\in [R].
$$
This inequality holds if and only if there exists an $p_{i\ell jr} \geq 0$ such that
	\begin{align*}
	    (2y_{ij})^2 \leq & (p_{i\ell jr} + \pi_{i\ell jr})^2 - (p_{i\ell jr} - \pi_{i\ell jr})^2 \qquad \forall i\in I, \forall \ell \in I_r, \forall j\in J, \forall r\in [R],\\
	    (2p_{i\ell jr})^2 \leq & \left(k_j + \frac{1}{h_{i\ell jr}}\right)^2 - \left(k_j - \frac{1}{h_{i\ell jr}}\right)^2 \qquad \forall i \in I, \forall \ell \in I_r ,\forall j\in J,\forall r\in [R].
	\end{align*}
It remains to incorporate the case that $y_{\ell j} = 0$. To this end, we revise the above two inequalities as
\begin{align*}
    &(2y_{ij})^2 \leq (p_{i\ell jr} + \pi_{i\ell jr} + 2(1 - y_{\ell j}))^2 - (p_{i\ell jr} - \pi_{i\ell jr})^2 \qquad \forall i\in I, \forall  \ell \in I_r, \forall j\in J,\forall r\in [R], \\
    &(2p_{i\ell jr})^2 \leq  \left(k_j + \frac{1}{h_{i\ell jr}}y_{\ell j}\right)^2 - \left(k_j - \frac{1}{h_{i\ell jr}}y_{\ell j}\right)^2\qquad \forall i\in I, \forall \ell \in I_r,\forall j\in J,\forall r\in [R].
\end{align*}
both of which are second-order conic representable. The corresponding standard-form SOC constrains are \eqref{dp:con2}--\eqref{dp:con3}.  This completes the proof.\Halmos
\end{proof}

\section{Computational study} \label{section6}
In this section,  we conduct the  computational study  to analyze the performance of the proposed  location-allocation problems with congestion under non-priority, static priority, and dynamic  priority disciplines. Since $M/G/1$ queuing system is applied to approximate  $M/G/k$  queue, we use the simulation method to evaluate the system performance. This approach combines optimization model and simulation,  which has been used in several studies \citep{140.1fujiwara1987ambulance,140.2goldberg1990validating,140harewood2002emergency}. Our research builds the simulators of the $M/G/k$ queuing system  based on  the three disciplines respectively.  The simulators are built in the  SimPy with Python. The optimal solutions of  the mathematical programming models together with other initial settings are input in the simulators. The total running time of each simulation process is set to be 30,000 time units to ensure the stable state of the queuing system. Then the average waiting time of each priority class at each center obtained from the simulation together with the travel time is used to analyze the performance.

In section \ref{static},  we investigate the impact of the number of drones and the weight of priorities on  response time  of the static priority problem and compare the performance between the non-priority and static priority system.
In section \ref{dynamic}, sensitivity analysis is conducted for the dynamic priority problem. The performance of the system is compared to that of static priority and non-priority system.

The deduced SOCP formulations are solved by Gurobi Optimizer version 9.0.1. All tests are performed on a personal computer with  2.9 GHz processor and 8 GB memory, running Windows 10 operating system. In all non-priority, static priority and dynamic priority problems, the coordinates of both the candidate facilities and demand nodes are randomly generated in a uniform distribution $U~(0,30)$. Drone speed is set to  80 km/h, and its endurance is  40 min. The default value of $w_1$ is set to $70\%$.

\subsection{Analysis of static priority problem}\label{static}
This section  gives the computational results of our  facility location-allocation problem with congestion and static priority. 
We explore the impacts of different coefficients on the performance of the system,
and evaluate the gap  between non-priority and static priority systems by comparing  response time and waiting time to the non-priority system.

 In all experiments, we use the instances comprised of 10 demand nodes with 2 priority classes  of materials at any node and 6 candidate service centers. The average arrival rate at demand node $i$ is randomly generated as $\lambda_i \sim U(0.6,1)$. The proportion for demands of priority class $1$ of each node $v_{i1}$  is produced by the uniform distribution of $(0,1)$, leading to $v_{i2}=1-v_{i1}$. The initial maximum total number of drones is set to  200.

\subsubsection{Impact of the number of drones}\label{6.1.1}
\ 
\newline
\indent
In this section, we investigate the impact of the number of drones on  response time in  the static priority system. Based on the aforementioned settings, we first solve the side model to get the  minimum number of drones $K^*$ associated with the stable queuing system. Then we run a set of experiments for (SP) using different values of $\alpha$ (0, 0.02, 0.05,  0.1, 0.2, 0.5, 1) to bound the  total number of drones, taking the value of $K^*(1+\alpha)$.
For each of the above settings, we can obtain the optimal decision of the location of the open centers, the allocation of the demands, and the deployment of the drones, all of which are further input into the simulator to access the system performance.

We select a typical instance to represent the results of both the optimization model and the simulation  in Figure \ref{s11}. The weighted expected response time $Z$ which is the objective function of  (SP)  goes down as the number of drones increases  in both situations. The values obtained from the optimization model  are a little higher than the simulation results,  but the difference is not that pronounced.
The weighted response time is extremely high when $\alpha$ is very small, and shows a sharp decrease as $\alpha$ rises. Finally, it becomes stable when $\alpha$ exceeds 0.2. The observation is associated with the fact that when  drones are particularly scarce, the waiting time of demands  takes a large  portion  in the response time. However, when drones become sufficient, the waiting time becomes really small even to $0$ while the travel time 
achieves the dominant role instead, leading to the stabilization of $Z$. The results suggest the unnecessity to distribute excessive drones at the centers which stands for a high cost. According to our experiments on majority instances, setting  $\alpha$ around 0.2   can decrease the response time significantly without incurring too much cost.

\begin{figure}[htbp]
\centering 
\includegraphics[width=3.5in]{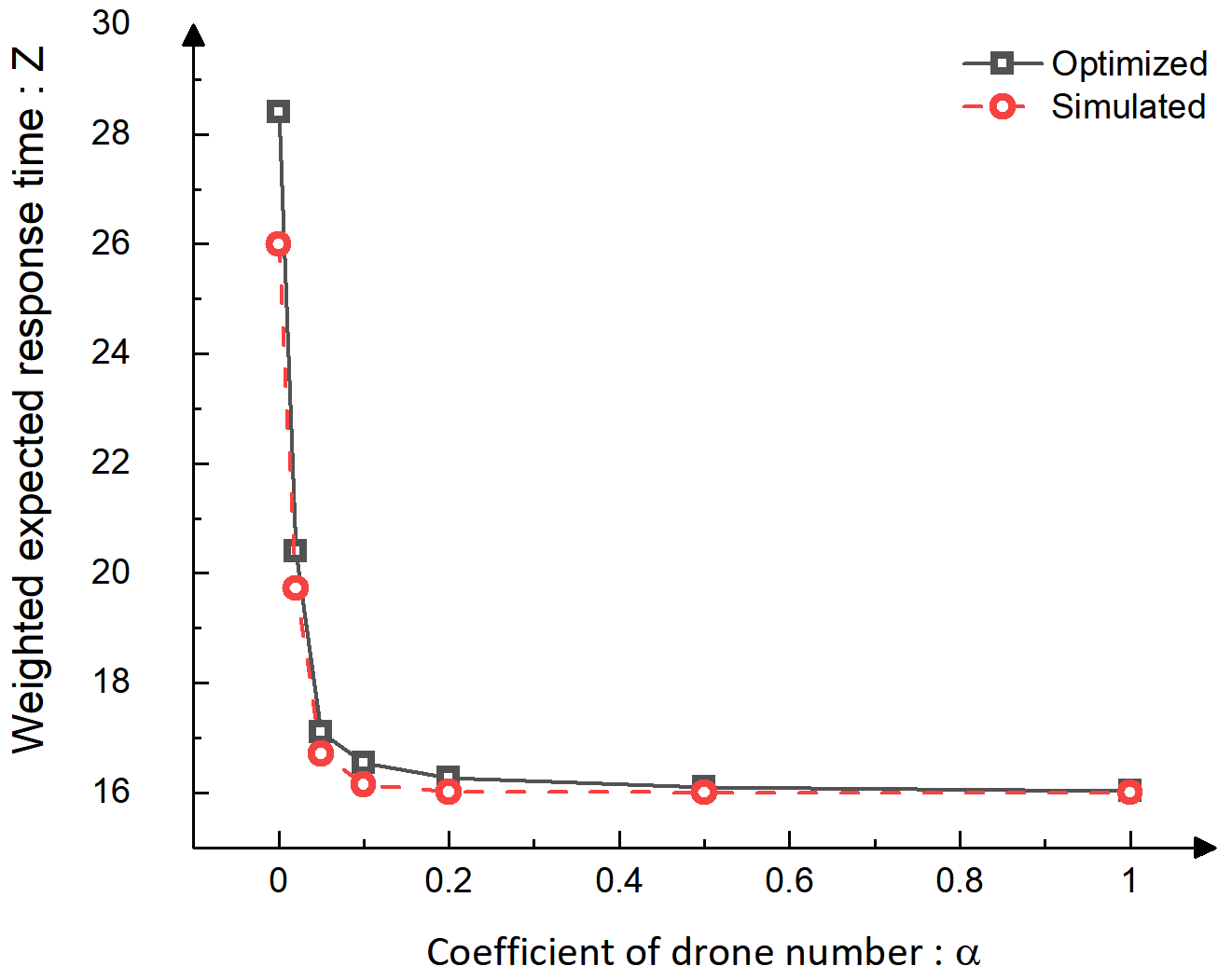} 
\caption{Weighted  expected response time $Z$ under different $\alpha$  (0, 0.02, 0.05,  0.1, 0.2, 0.5, 1)} 
\label{s11} 
\end{figure}

To analyze the impact of drone number on each priority class, Figure \ref{s12} provides the results for 10 randomly generated instances under each value of $\alpha$.  Without specific denoting, the results are obtained through the optimization-simulation approach.  $Z_r$  is the maximum  expected  response time among all demand nodes with priority $r$.  The mean value of all instances  is shown as the middle line in the area between  the upper and lower quartiles. The variation trend of class 2 demand is similar to that of $Z$, which declines dramatically at first and  becomes stable later on. However, there is only a slight decrease in class 1 demand throughout. 
The reason is that  in the static priority queuing system, the demands from  class 1 have strict priority over the others. Once there is an available drone, the  highest priority demand will be delivered immediately, even the relatively small number of drones doesn't have a big influence. However, when drones are insufficient, it will take a long time to deliver all existing and successively 
emerging demands of class 1 before a drone is able to serve a  class 2 demand. Besides, it is noted that the maximum  response time of class 2 is always larger than  that of class 1 and finally becomes equal when   drones are sufficient.

\begin{figure}[!htb]
\centering 
\includegraphics[width=3.5in]{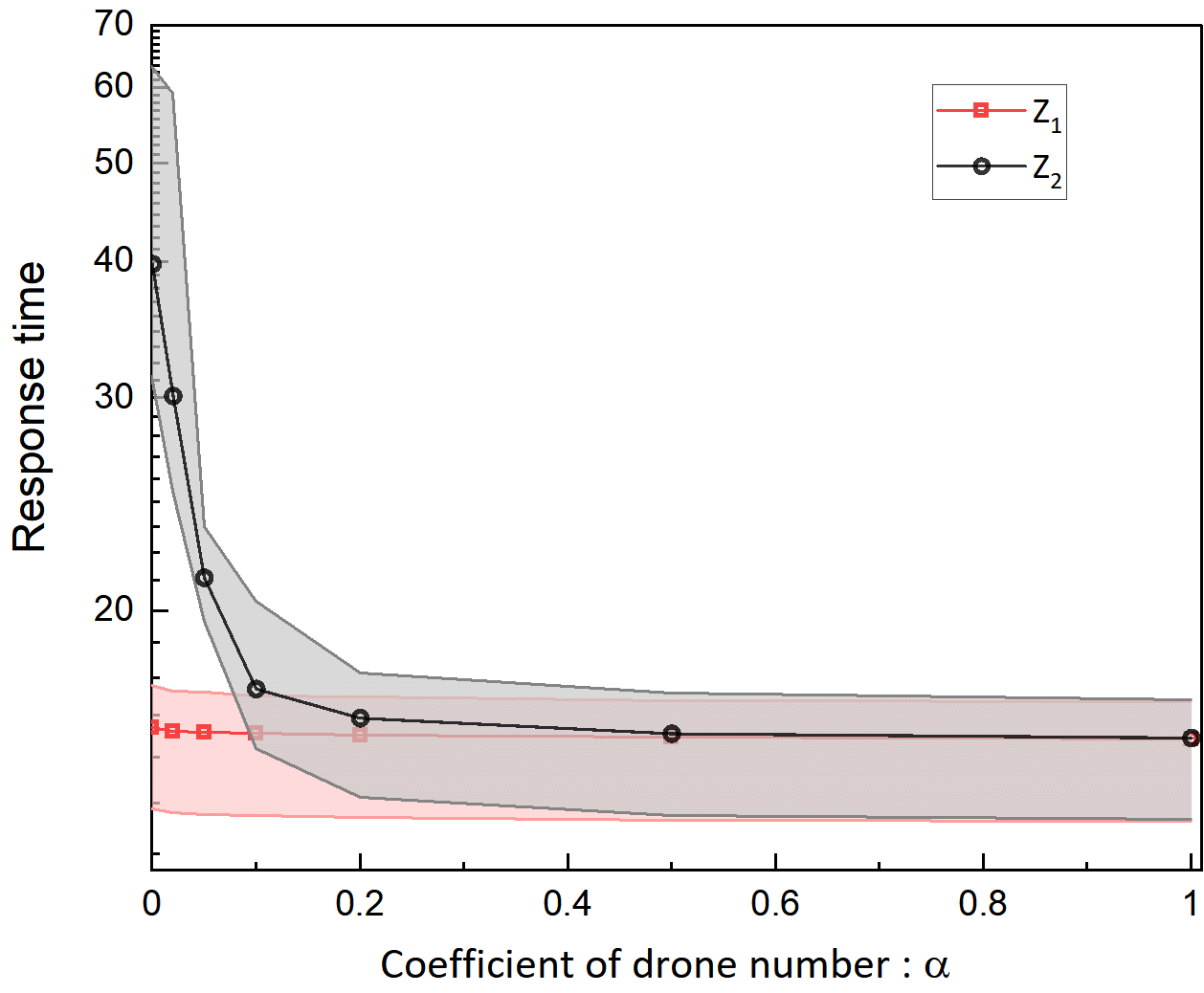} 
\caption{Maximum  expected  response time of class 1 and class 2 demands with different $\alpha$ (0, 0.02, 0.05,  0.1, 0.2, 0.5, 1) } 
\label{s12} 
\end{figure}

\subsubsection{Impact of the weight coefficients}
\ 
\newline
\indent
In this section, we study the impact of the weight of priorities $w_r$  on  system performance. In addition to  the maximum response time $Z_r$, we also evaluate the performance of total response time  $sumZ_r=\Sigma_{i\in I_r}\lambda_iv_{ir}y_{ij}^r(t_{ij}+W_{jr}) $, maximum  waiting time $W_r$, and total waiting time  $sumW_r=\Sigma_{i\in I_r}\lambda_iv_{ir}y_{ij}^rW_{jr}$ for each priority class.

We generate 10 instances randomly, and similar results are observed from all instances. The results of a typical instance are shown in Table \ref{table1}.  As expected, we see an increasing  trend of $Z_1$ and a decreasing  trend of $Z_2$ as the coefficient $w_1$ goes down, but the trends are not remarkable. 
We observe that other measures are also insensitive to $w_r$ except for some extreme parameters, such as $w_1=100$.
Indeed,  the weight of priorities $w_r$ can hardly influence the performance of any priority class in the static system since higher priority demands have already been respected.  
The performance of class 1 demands usually maintains a steady state, as illustrated in section \ref{6.1.1}, it is less affected by the number of drones.
However, $w_r$ matters in a dynamic priority system, as we will later see in Table \ref{table2}.

\begin{table}[htbp]
\renewcommand\arraystretch{0.6}
  \centering
  \caption{Static-priority system performance by varying $w_1$ with $\alpha=$ 0.1, 0.2, and 0.5}
    \begin{tabular}{cccccccccc}
    \toprule
    \multirow{2}[4]{*}{$\alpha$} & \multirow{2}[4]{*}{$w_1(\%)$} & \multicolumn{2}{c}{Response time} & \multicolumn{2}{c}{Total response time} & \multicolumn{2}{c}{Waiting time} & \multicolumn{2}{c}{Total waiting time} \\
\cmidrule(rl){3-4} \cmidrule(rl){5-6}\cmidrule(rl){7-8}\cmidrule(rl){9-10}         &       & $Z_1$    & $Z_2$    & $sumZ_1$ & $sumZ_2$ & $W_1$    & $W_2$    & $sumW_1$ & $sumW_2$ \\
    \midrule
    0.1   & 100   & 11.185  & 15.882  & 34.403  & 52.874  & 0.603  & 6.882  & 1.903  & 20.374  \\
          & 99    & 11.243  & 15.882  & 34.178  & 54.199  & 0.468  & 6.882  & 1.678  & 21.699  \\
          & 70     & 11.269  & 14.647  & 34.149  & 51.475  & 0.468  & 6.647  & 1.649  & 18.975  \\
          & 50     & 11.269  & 14.647  & 34.149  & 51.475  & 0.468  & 6.647  & 1.649  & 18.975  \\
          & 30     & 11.269  & 14.647  & 34.149  & 51.475  & 0.468  & 6.647  & 1.649  & 18.975  \\
          & 10     & 11.269  & 14.647  & 34.149  & 51.475  & 0.468  & 6.647  & 1.649  & 18.975  \\
          & 0     & 11.269  & 14.647  & 34.149  & 51.475  & 0.468  & 6.647  & 1.649  & 18.975  \\
    \midrule
    0.2   & 100   & 11.110  & 15.882  & 34.291  & 52.094  & 0.603  & 6.882  & 1.791  & 19.594  \\
          & 99    & 11.155  & 13.405  & 34.082  & 44.721  & 0.477  & 5.405  & 1.582  & 12.221  \\
          & 70     & 11.185  & 11.738  & 33.935  & 41.275  & 0.394  & 2.974  & 1.335  & 9.775  \\
          & 50     & 11.185  & 11.738  & 33.935  & 41.275  & 0.394  & 2.974  & 1.335  & 9.775  \\
          & 30     & 11.185  & 11.738  & 33.935  & 41.275  & 0.394  & 2.974  & 1.335  & 9.775  \\
          & 10     & 11.185  & 11.738  & 33.935  & 41.275  & 0.394  & 2.974  & 1.335  & 9.775  \\
          & 0     & 11.185  & 11.738  & 33.935  & 41.275  & 0.394  & 2.974  & 1.335  & 9.775  \\
    \midrule
    0.5   & 100   & 11.033  & 25.071  & 33.879  & 63.587  & 0.430  & 12.071  & 1.379  & 30.087  \\
          & 99    & 11.053  & 11.084  & 33.598  & 40.377  & 0.350  & 2.974  & 1.098  & 7.877  \\
          & 70     & 11.053  & 11.084  & 33.598  & 40.377  & 0.350  & 2.974  & 1.098  & 7.877  \\
          & 50     & 11.053  & 11.084  & 33.598  & 40.377  & 0.350  & 2.974  & 1.098  & 7.877  \\
          & 30     & 11.053  & 11.084  & 33.598  & 40.377  & 0.350  & 2.974  & 1.098  & 7.877  \\
          & 10     & 11.053  & 11.084  & 33.598  & 40.377  & 0.350  & 2.974  & 1.098  & 7.877  \\
          & 0     & 16.521  & 11.057  & 36.326  & 40.349  & 3.521  & 2.974  & 2.826  & 7.849  \\
    \bottomrule
    \end{tabular}%
  \label{table1}%
\end{table}%

In order to reveal the relative changes of the two priority demands more intuitively, we use Figure \ref{s31} to display the  ratio of the total response time and waiting time of class 1 to class 2 demand by varying $w_1$ with $\alpha=$ 0.1, 0.2, and 0.5.
It is noted that the total waiting time of class 1 demands is relatively small 
compared to that of class 2, where the ratio is below $10\%$ with $\alpha=0.1$ and a little higher with larger $\alpha$. The decreasing trend of the ratio is not noticeable with  $w_1$ except at the extreme value.  The ratio of total response time is larger than that of total waiting time, while the variation trend is similar. Overall, the weight coefficient $w_r$  has less impact on the response time and waiting time  of both classes,  reflecting the nature of static priority. 

\begin{figure}[!htbp]
\centering
\subfigure[$\alpha=0.1$]{
\begin{minipage}[t]{0.33\linewidth}
\centering
\includegraphics[width=2in]{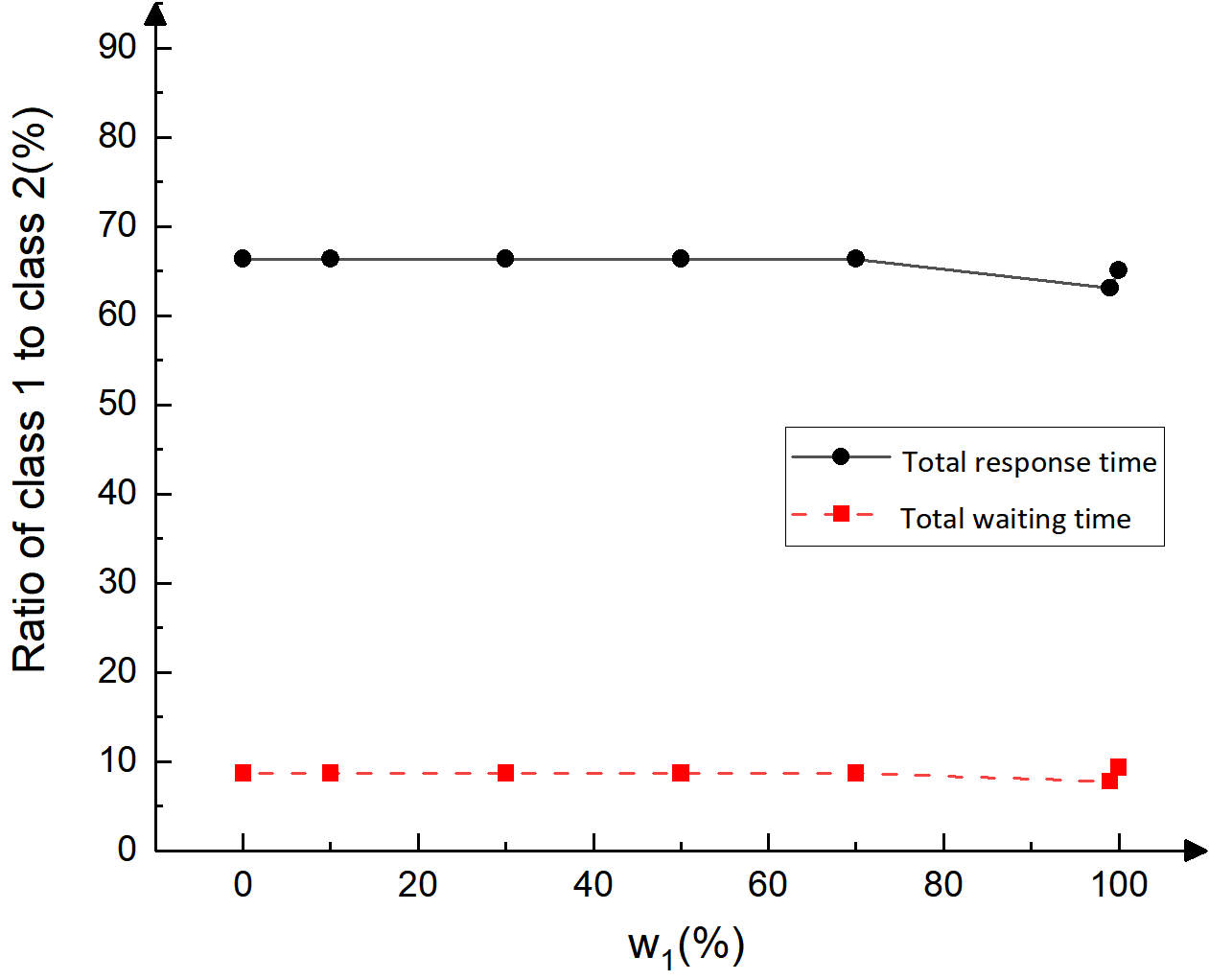}
\end{minipage}%
}%
\subfigure[$\alpha=0.2$]{
\begin{minipage}[t]{0.33\linewidth}
\centering
\includegraphics[width=2in]{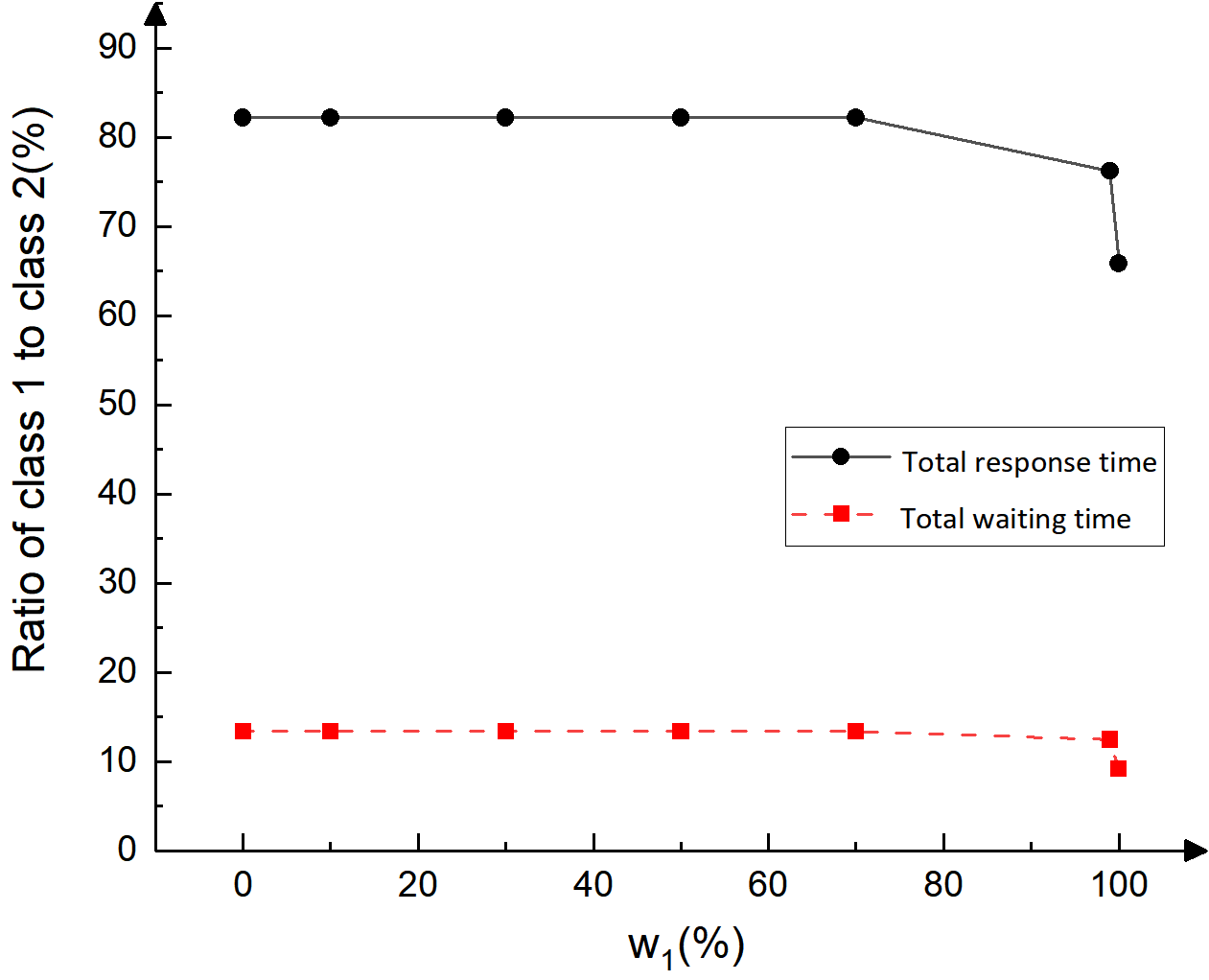}
\end{minipage}%
}%
\subfigure[$\alpha=0.5$]{
\begin{minipage}[t]{0.33\linewidth}
\centering
\includegraphics[width=2in]{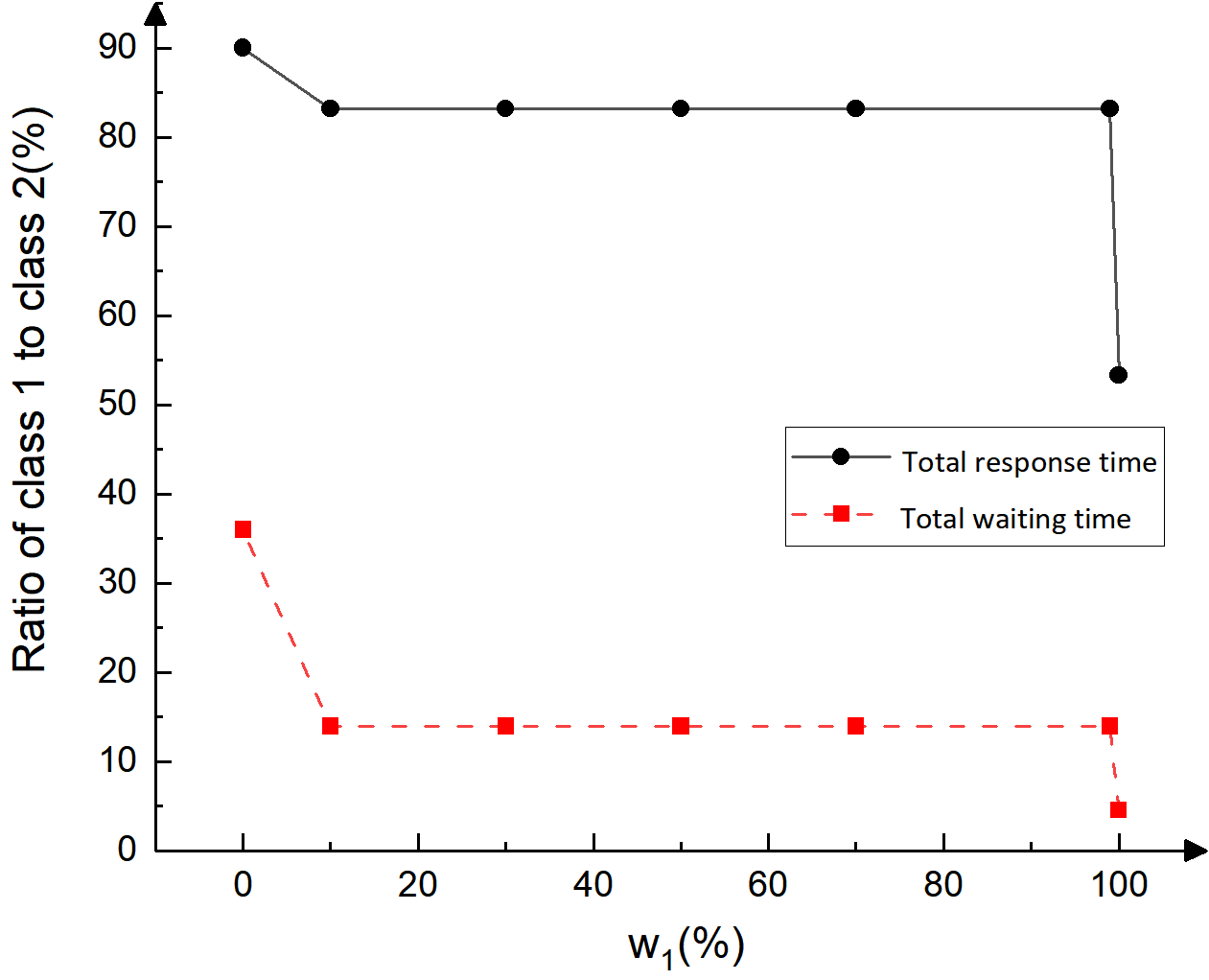}
\end{minipage}
}%
\centering
\caption{ Ratio of total response time and total waiting time of class 1 to class 2 demand by varying $w_1$ with $\alpha=$ 0.1, 0.2, and 0.5}
\label{s31}
\end{figure}

\subsubsection{Comparison with the non-priority system}\label{npsp}
\ 
\newline
\indent
In this section, we compare the performance of our  congested facility location-allocation system with  static priority to that with non-priority. Each demand is  given the same priority class in the non-priority system as in the static priority system so as to get the same objective function.
Figure \ref{s3}  presents the box plots of the gap of total response time and waiting time between the static priority system and non-priority system, which is calculated as follows:
\begin{equation*}
    Gap=\frac{T_{SP}-T_{NP}}{T_{NP}}\times 100\%,
\end{equation*}
where $T_{SP}$ represents the total response time or waiting time in the static priority system and $T_{NP}$ for that in the non-priority system. 20 instances are randomly generated with the parameter setting $\alpha=0.1$ and $w_1=70\%$. 

According to the negative values of $Gap$ for total response time and total waiting time of class 1 demand and the positive values of  class 2 demands, we can conclude  that introducing static priority discipline into the delivery system can help decrease both the response time and waiting time of the higher priority demands but incur extra time for lower priority demands. Therefore, when there are  demands of urgent need while others are not that emergent, it is more reasonable to apply the static priority.
It is also worth mentioning that  the gap in waiting time is more substantial than that in response time, the mean absolute value of  which are about  $15\%$ and $70\%$ respectively. This manifests that although static priority discipline has a considerable impact on waiting time, travel time also accounts for an essential part of the final response time. 

\begin{figure}[!htb]
\centering 
\includegraphics[width=3.5in]{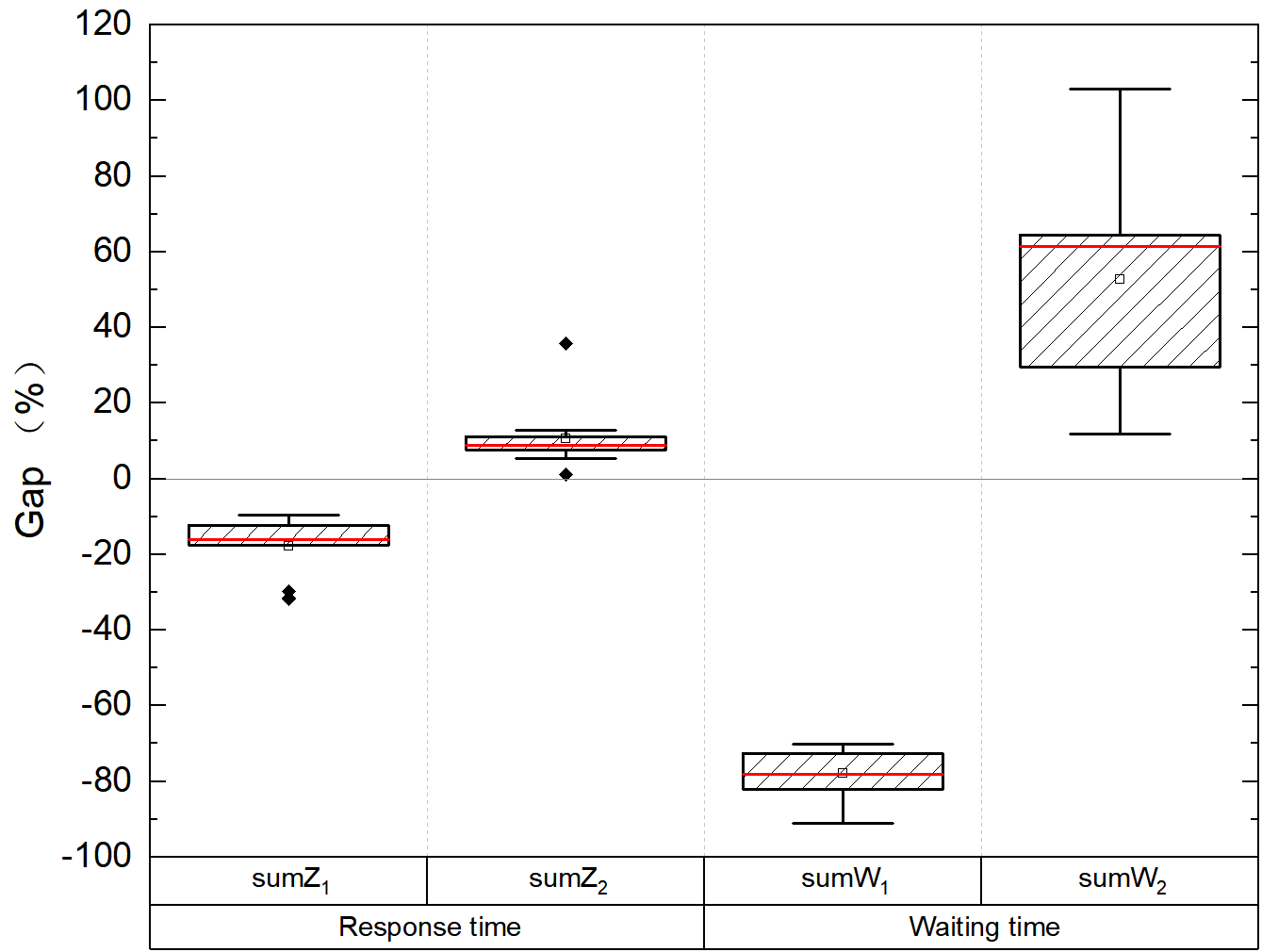} 
\caption{Gap of total response time and total waiting time between the static priority system and non-priority system ($\alpha=0.1$)} 
\label{s3} 
\end{figure}

\subsection{Analysis of dynamic priority problem}\label{dynamic}
In this section, we present the computational results of our  congested facility location-allocation problem with  dynamic priority. We investigate  the impact of  the total number of drones on the system performance under different settings of initial priority-class gap $\Delta a_{ij}$.   We also conduct an analysis of the weight coefficient $w_r$ in the  (DP) . The performance of the  dynamic system is compared with that of static priority and non-priority queuing-location system.

The instances are random generated with 11 demand nodes and 6 candidate service centers with the distribution mentioned before. Two priority classes are considered.  We randomly choose 6 demand nodes as class 1 demands, and the others belong to class 2. The average arrival rate at  demand node $i$  is randomly generated as $\lambda_i \sim U(0.1,0.5)$. The number of drones used in the system is limited to 200. Since only two priority classes are studied, we use $\Delta a$ instead of $\Delta a_{12}$ for simplicity.

\subsubsection{Impact of the number of drones}\label{6.2.1}
\ 
\newline
\indent
In this section, we exam the impact of the number of drones on the response time for the dynamic priority system. The side model is  solved first to get the  minimum number of drones $K^*$. Then  (DP)  is solved to optimality with  different number of drones by varying coefficient $\alpha \in \{0, 0.02, 0.05, 0.1, 0.2, 0.5, 1\}$ under different  initial priority-class gap $\Delta a\in \{ 3, 10, 20\}$. The impacts on the response time are analyzed through the simulation process with respect to the optimal results from   (DP) .

Figure \ref{d12} represents the results of a typical instance of both the optimization model and the simulation, which illustrates the impact of the total number of drones on the weighted expected response time $Z$  with different $\Delta a$. The decreasing trends of $Z$ under all settings of $\Delta a$ in both situations (optimization and simulation) are similar to those in the static priority queuing system, which fall  rapidly at the beginning and then become stable till the end. This result also implies that it is appropriate to deploy drones with $\alpha$ around 0.2 in the whole system. Since the upper bound is used in the optimization model of the  dynamic problems, the value of $Z$ is slightly  higher than that obtained from the simulation.  
Besides, we find that the curves with larger $\Delta a$ lie above or align with those with smaller $\Delta a$. This observation indicates that a larger initial priority gap between two classes  will lead to a larger weighted expected response time in general. All the three lines converge to a certain value as $\alpha$ approaches 1, which is associated with the fact that  the waiting time is extremely small and only travel time matters for the response time when drones are exceedingly sufficient.
\begin{figure}[!htbp]
    \centering
    \includegraphics[width=3.5in]{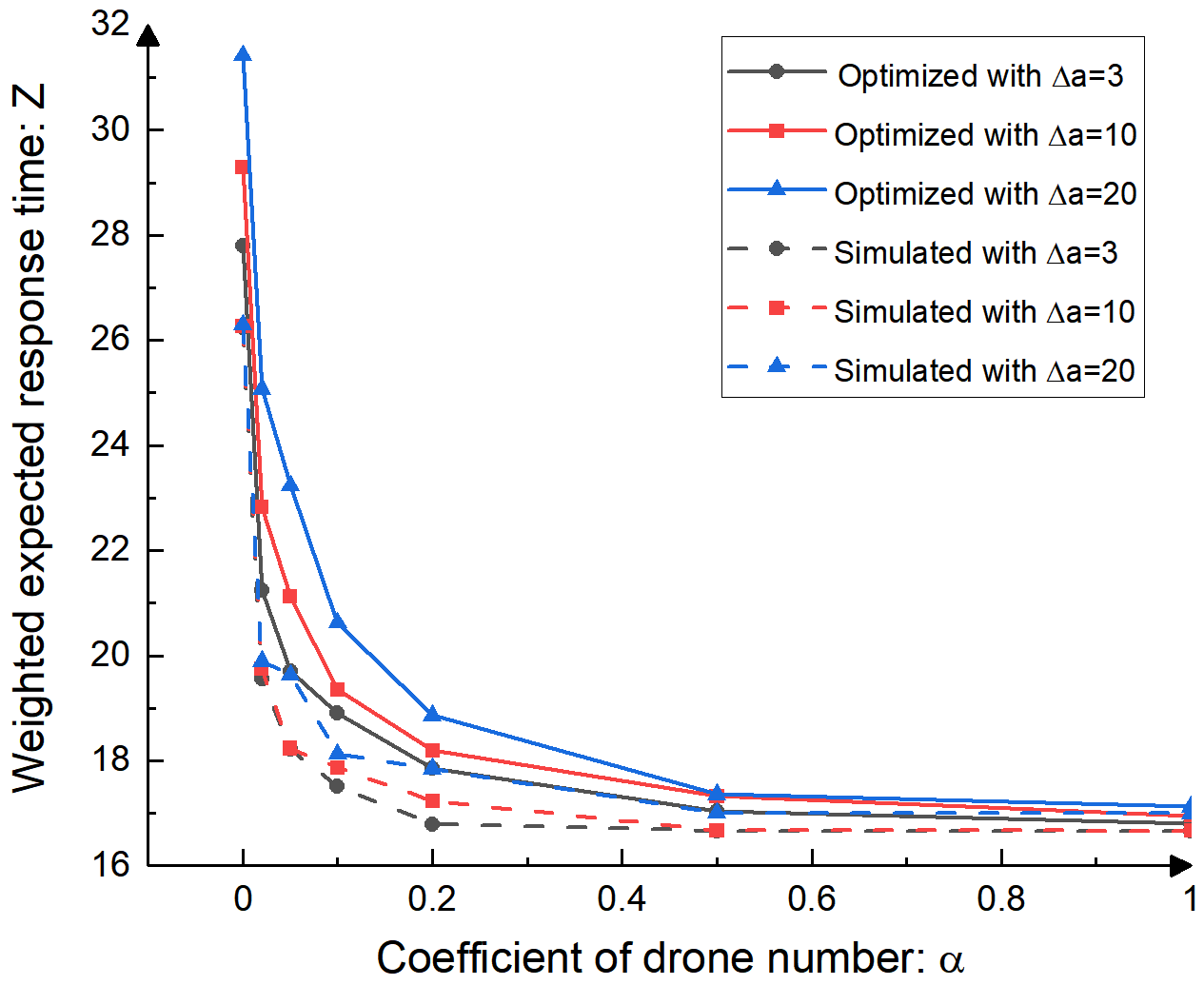}
    \caption{The impact of total number of drones on the weighted expected response time $Z$  with  different $\Delta a$ (3, 10, 20) }
    \label{d12}
\end{figure}

The changes in $Z_r$, which is  the expected maximum response time  of class $r$ demands by varying  coefficient $\alpha$ of drone number when $\Delta a =3$, 10 and 20 can be seen in Figure \ref{d11}. 10 instances are randomly developed for each setting.  The mean values, the upper and the lower quartiles of $Z_r$ are displayed in the graphs. 
As expected, the response times of both classes decline as the number of drones increases and finally converge to a certain value in all settings. 
Besides, with the growth of $\Delta a$, the downward trend of $Z_1$ becomes flatter, indicating the  declining influence of $\alpha$  on the response time of high priority demands. 
While the gap between $Z_1$ and $Z_2$ becomes larger as $\Delta a$ grows. In order to show this phenomenon more directly, we present the  ratio of $Z_2/Z_1$ of a  typical instance in Figure \ref{d13}.   The ratio with a larger  $\Delta a$ is always higher  than that of a smaller value. The lower the ratio, the smaller the difference between the two classes.
As $\Delta a$ increases, the initial priority value plays a more prominent role than time in the final priority value, leading to a larger gap  of the response time between the two classes,  which  approaches the situation in the static system.

\begin{figure}[!htbp]
\centering
\subfigure[$\Delta a=3$]{
\begin{minipage}[t]{0.33\linewidth}
\centering
\includegraphics[width=2in]{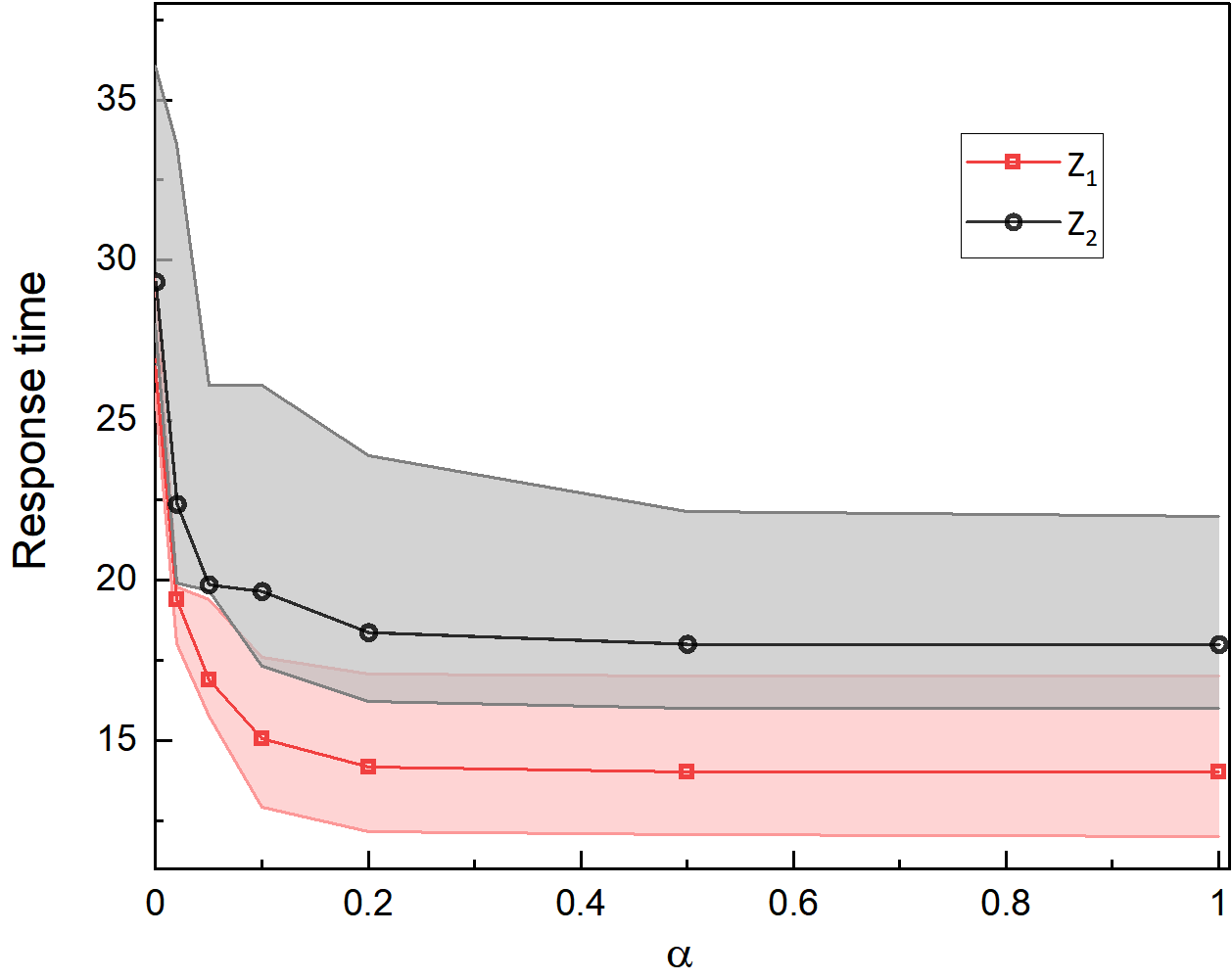}
\end{minipage}%
}%
\subfigure[$\Delta a=10$]{
\begin{minipage}[t]{0.33\linewidth}
\centering
\includegraphics[width=2in]{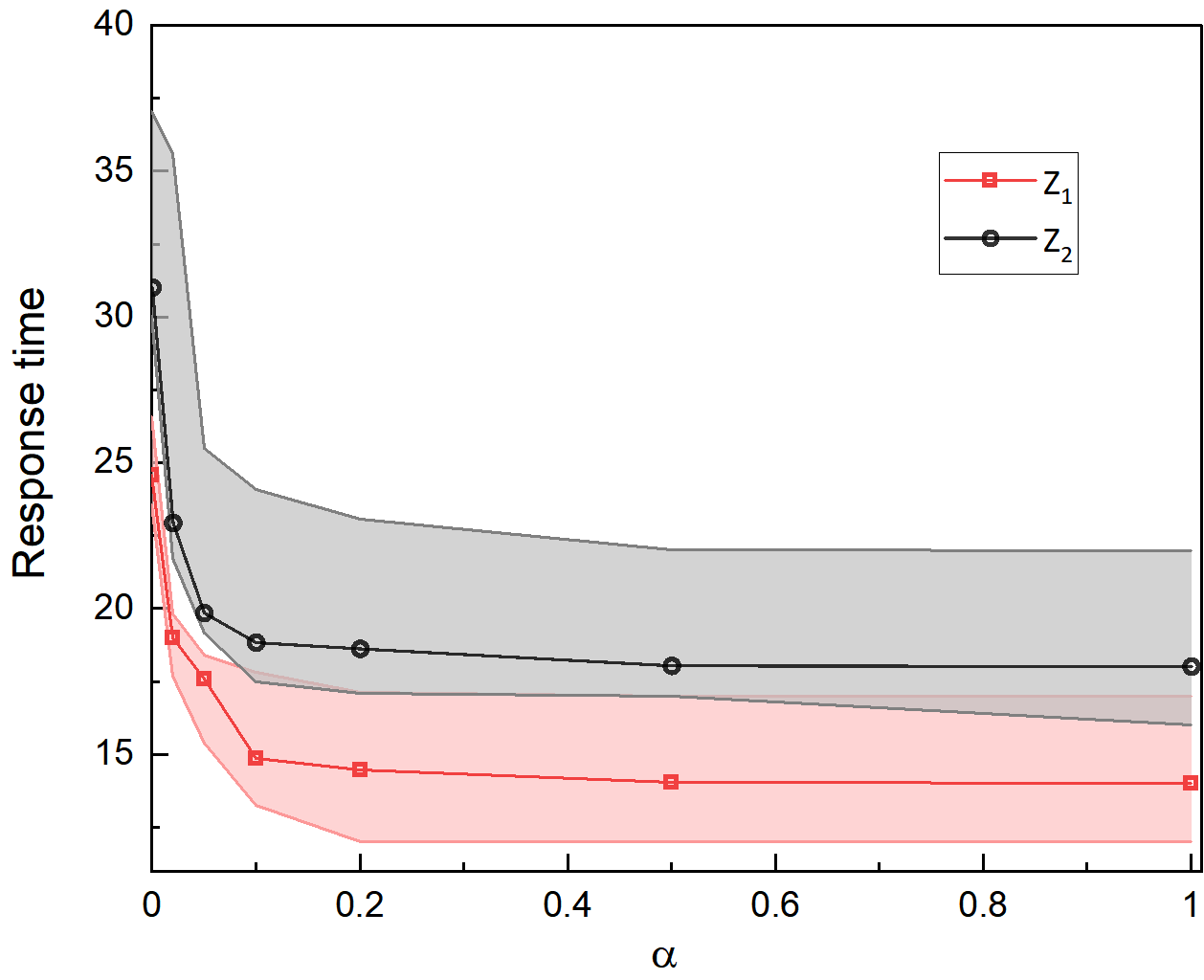}
\end{minipage}%
}%
\subfigure[$\Delta a=20$]{
\begin{minipage}[t]{0.33\linewidth}
\centering
\includegraphics[width=2in]{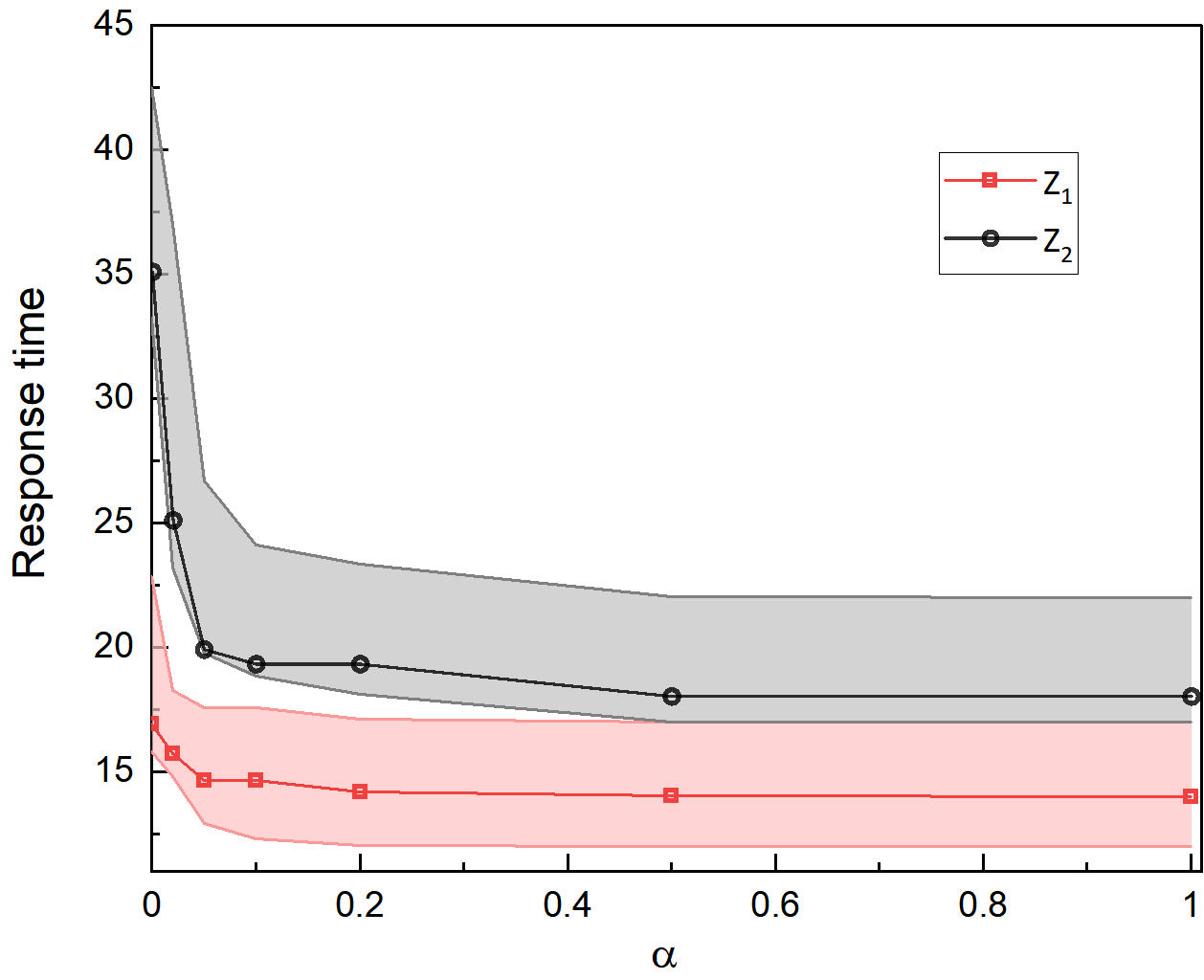}
\end{minipage}
}%
\centering
\caption{ Maximum expected response time of class 1 and class 2 demands with different $\alpha$ ($\Delta a =$3, 10,and 20)}
\label{d11}
\end{figure}

\begin{figure}[!htbp]
    \centering
    \includegraphics[width=3.5in]{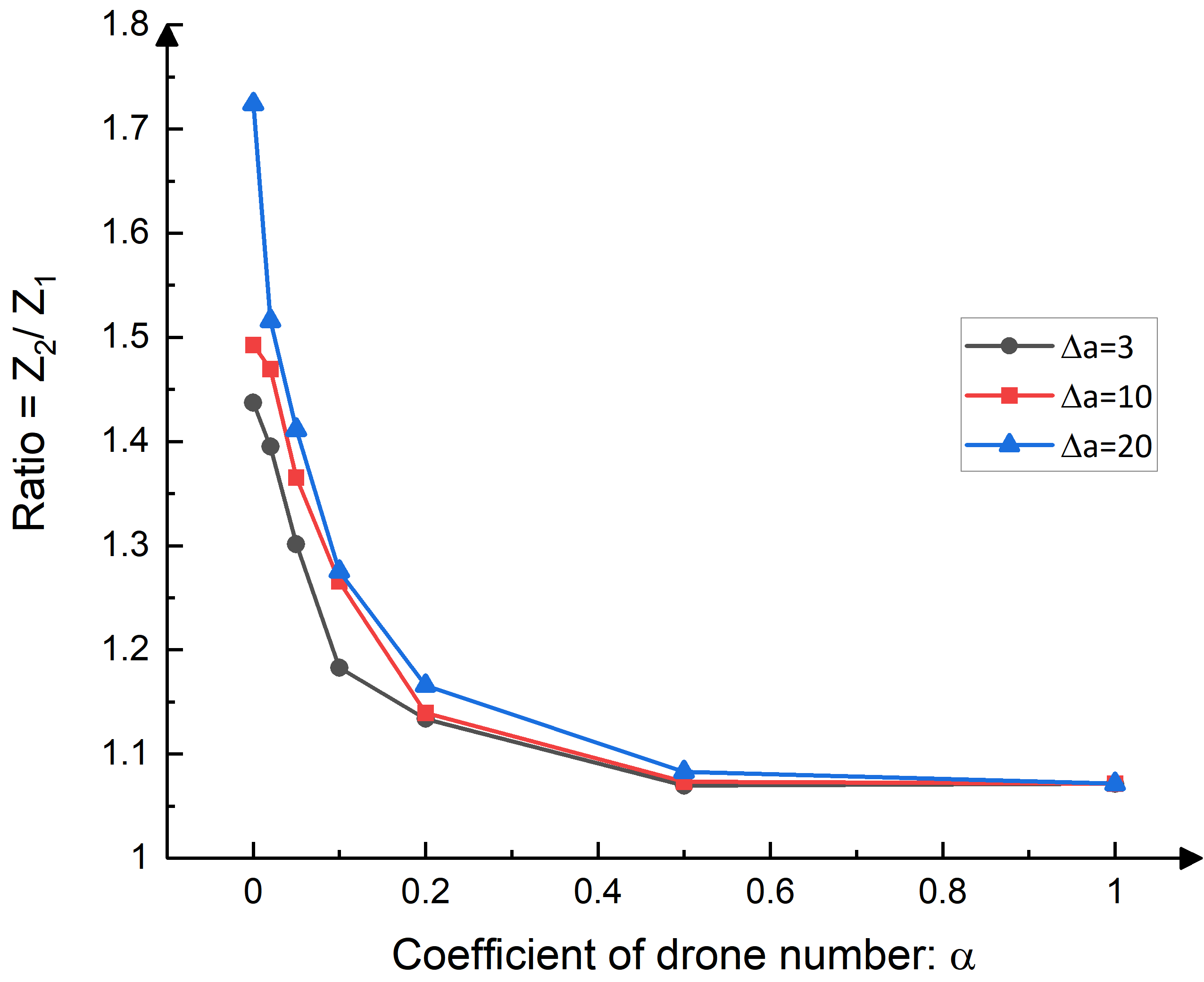}
    \caption{Ratio of the maximum expected response time of class 1 demand to class 2 demand}
    \label{d13}
\end{figure}

\subsubsection{Impact of the weight coefficients}
\ 
\newline
\indent
In this section, we study the impact of the weight of priorities $w_r$   on   the maximum  response time $Z_r$ , the total response time   $sumZ_r=\Sigma_{i\in I_r}\lambda_iy_{ij}(t_{ij}+W_{jr}) $,  the maximum  waiting time $W_r$, and the total waiting time  $sumW_r=\Sigma_{i\in I_r}\lambda_iy_{ij}W_{jr}$.

In Table \ref{table2}, we report the results of a typical instance with  $\Delta a=3,10$, and $20$ and  $\alpha=0.1,0.2$, and $0.5$. Similar to the results in the  static priority system, the response time of class 2 changes in the same direction as $w_1$, while class 1 is the opposite. The trends of $sumZ_r$, $W_r$, and $sumW_r$  correspond to $Z_r$ in general except for some extreme values. Overall,  the values of all indicators  tend to be improved with the rise of   $\alpha$  consistent with the conclusions in section \ref{6.2.1}.

\begin{table}[htbp]
\renewcommand\arraystretch{0.6}
  \centering
  \caption{System performance by varying $w_1$ with $\alpha=$ {0.1, 0.2,  0.5} and $\Delta a=${3, 10, 20}}
    \begin{tabular}{ccccccccccc}
    \toprule
    \multirow{2}[4]{*}{$\Delta a$} & \multirow{2}[4]{*}{$\alpha$} & \multirow{2}[4]{*}{$w_1$(\%)} & \multicolumn{2}{c}{Response time} & \multicolumn{2}{c}{Total response time} & \multicolumn{2}{c}{Waiting time} & \multicolumn{2}{c}{Total waiting time} \\
\cmidrule(rl){4-5} \cmidrule(rl){6-7}\cmidrule(rl){8-9}\cmidrule(rl){10-11}              &       &       & $Z_1$    & $Z_2$    & $sumZ_1$ & $sumZ_2$ & $W_1$    & $W_2$    & $sumW_1$ & $sumW_2$ \\
    \midrule
    3     & 0.1   & 100   & 14.446  & 21.476  & 18.000  & 24.011  & 3.300  & 9.911  & 4.609  & 12.684  \\
          &       & 70    & 14.943  & 17.678  & 16.709  & 18.453  & 2.009  & 3.753  & 1.462  & 2.678  \\
          &       & 50    & 14.943  & 17.678  & 16.709  & 18.453  & 2.009  & 3.753  & 1.462  & 2.678  \\
          &       & 30    & 16.467  & 16.130  & 17.943  & 17.244  & 3.243  & 2.544  & 2.467  & 3.960  \\
          &       & 0     & 16.467  & 16.130  & 17.943  & 17.244  & 3.243  & 2.544  & 2.467  & 3.960  \\
    \midrule
    3     & 0.2   & 100   & 14.066  & 21.476  & 17.544  & 23.775  & 2.844  & 9.675  & 4.609  & 12.684  \\
          &       & 70    & 14.228  & 16.856  & 15.539  & 16.203  & 0.839  & 2.103  & 0.941  & 1.856  \\
          &       & 50    & 14.446  & 15.578  & 16.053  & 16.162  & 1.353  & 2.062  & 1.363  & 2.394  \\
          &       & 30    & 14.943  & 15.235  & 16.650  & 16.218  & 1.950  & 2.118  & 1.363  & 2.394  \\
          &       & 0     & 16.467  & 15.097  & 18.478  & 16.783  & 3.778  & 2.683  & 2.467  & 3.960  \\
    \midrule
    3     & 0.5   & 100   & 14.000  & 21.476  & 17.465  & 23.733  & 2.765  & 9.633  & 4.609  & 12.684  \\
          &       & 70    & 14.010  & 15.038  & 15.530  & 15.564  & 0.830  & 1.464  & 1.363  & 2.394  \\
          &       & 50    & 14.035  & 15.011  & 15.077  & 15.861  & 0.377  & 1.161  & 0.557  & 1.262  \\
          &       & 30    & 14.122  & 15.001  & 15.080  & 14.671  & 0.380  & 0.571  & 0.389  & 0.820  \\
          &       & 0     & 14.943  & 15.000  & 16.066  & 15.132  & 1.366  & 1.032  & 0.943  & 1.800  \\
    \midrule
    10    & 0.1   & 100   & 14.370  & 21.950  & 16.770  & 24.546  & 2.070  & 10.446  & 2.711  & 12.684  \\
          &       & 70    & 14.764  & 18.684  & 15.892  & 20.493  & 1.192  & 6.393  & 0.764  & 12.684  \\
          &       & 50    & 14.764  & 18.684  & 15.892  & 20.493  & 1.192  & 6.393  & 0.764  & 12.684  \\
          &       & 30    & 14.764  & 18.684  & 15.892  & 20.493  & 1.192  & 6.393  & 0.764  & 12.684  \\
          &       & 0     & 16.077  & 16.557  & 17.650  & 18.402  & 2.950  & 4.302  & 2.077  & 5.878  \\
    \midrule
    10    & 0.2   & 100   & 14.060  & 21.950  & 16.399  & 24.209  & 1.699  & 10.109  & 2.711  & 12.684  \\
          &       & 70    & 14.201  & 16.557  & 15.398  & 16.822  & 0.698  & 2.722  & 0.762  & 2.674  \\
          &       & 50    & 14.370  & 15.578  & 15.601  & 16.437  & 0.901  & 2.337  & 0.762  & 2.674  \\
          &       & 30    & 14.764  & 15.578  & 15.791  & 15.773  & 1.091  & 1.673  & 0.764  & 2.679  \\
          &       & 0     & 14.764  & 15.295  & 15.675  & 15.858  & 0.975  & 1.158  & 0.764  & 2.679  \\
    \midrule
    10    & 0.5   & 100   & 14.000  & 21.950  & 16.327  & 24.160  & 1.627  & 10.060  & 2.711  & 12.684  \\
          &       & 70    & 14.037  & 15.038  & 17.114  & 14.672  & 0.514  & 3.572  & 0.350  & 0.872  \\
          &       & 50    & 14.060  & 15.005  & 14.947  & 14.675  & 0.247  & 0.575  & 0.291  & 0.872  \\
          &       & 30    & 14.060  & 15.005  & 14.947  & 14.675  & 0.247  & 0.575  & 0.291  & 0.872  \\
          &       & 0     & 14.370  & 15.000  & 15.253  & 15.586  & 0.553  & 0.886  & 0.370  & 1.287  \\
    \midrule
    20    & 0.1   & 100   & 14.344  & 22.209  & 16.085  & 24.818  & 1.385  & 10.718  & 1.621  & 12.684  \\
          &       & 70    & 14.649  & 18.684  & 15.718  & 20.680  & 1.018  & 6.580  & 0.649  & 12.684  \\
          &       & 50    & 14.649  & 18.684  & 15.718  & 20.680  & 1.018  & 6.580  & 0.649  & 12.684  \\
          &       & 30    & 14.649  & 18.684  & 15.718  & 20.680  & 1.018  & 6.580  & 0.649  & 12.684  \\
          &       & 0     & 14.649  & 17.889  & 15.728  & 19.144  & 1.028  & 4.444  & 0.649  & 3.257  \\
    \midrule
    20    & 0.2   & 100   & 14.060  & 22.209  & 15.744  & 24.443  & 1.044  & 10.343  & 1.621  & 12.684  \\
          &       & 70    & 14.195  & 16.174  & 15.089  & 16.301  & 0.389  & 1.601  & 0.259  & 1.174  \\
          &       & 50    & 14.195  & 16.174  & 15.089  & 16.301  & 0.389  & 1.601  & 0.259  & 1.174  \\
          &       & 30    & 14.195  & 16.174  & 15.089  & 16.301  & 0.389  & 1.601  & 0.259  & 1.174  \\
          &       & 0     & 14.649  & 15.295  & 15.537  & 16.031  & 0.837  & 1.331  & 0.649  & 3.257  \\
    \midrule
    20    & 0.5   & 100   & 14.000  & 17.889  & 14.950  & 18.167  & 0.250  & 3.467  & 0.416  & 2.889  \\
          &       & 70    & 14.018  & 15.183  & 18.123  & 14.931  & 1.623  & 0.231  & 2.669  & 0.183  \\
          &       & 50    & 14.034  & 15.097  & 16.077  & 14.272  & 1.377  & 0.172  & 4.350  & 0.149  \\
          &       & 30    & 14.034  & 15.092  & 19.437  & 14.835  & 2.937  & 0.135  & 4.350  & 0.092  \\
          &       & 0     & 14.344  & 15.000  & 15.286  & 15.048  & 0.586  & 0.948  & 0.344  & 1.416  \\
    \bottomrule
    \end{tabular}%
  \label{table2}%
\end{table}%

To show the impact of $w_r$ on each priority class more intuitively, we display the ratios  of total waiting time and total response time  of class 1 versus class 2  with $\alpha=0.1$   in Figure \ref{d21}.  The growth of $w_1$  leads to a decreasing trend in the ratio of both  waiting time and response time. This effect is especially pronounced with smaller $\Delta a$ where the preferential treatment of class 1 demand is not that strict. We observe that the ratio of the total response time is larger with smaller  $\Delta a$ generally, so is the total waiting time. This can be explained that larger $\Delta a$  means more significant priority-class differences, where class 1 demand is  more prioritized, yielding a larger gap between the two classes. Overall, the weight coefficients have a great impact with  small $\Delta a$ and a slight impact with  large $\Delta a$ on the system performance. We can adjust the importance of each priority class by using different $w_r$ in dynamic priority system to match the actual situation.

\begin{figure}[!htbp]
\centering
\subfigure[$\Delta a=3$]{
\begin{minipage}[t]{0.33\linewidth}
\centering
\includegraphics[width=2in]{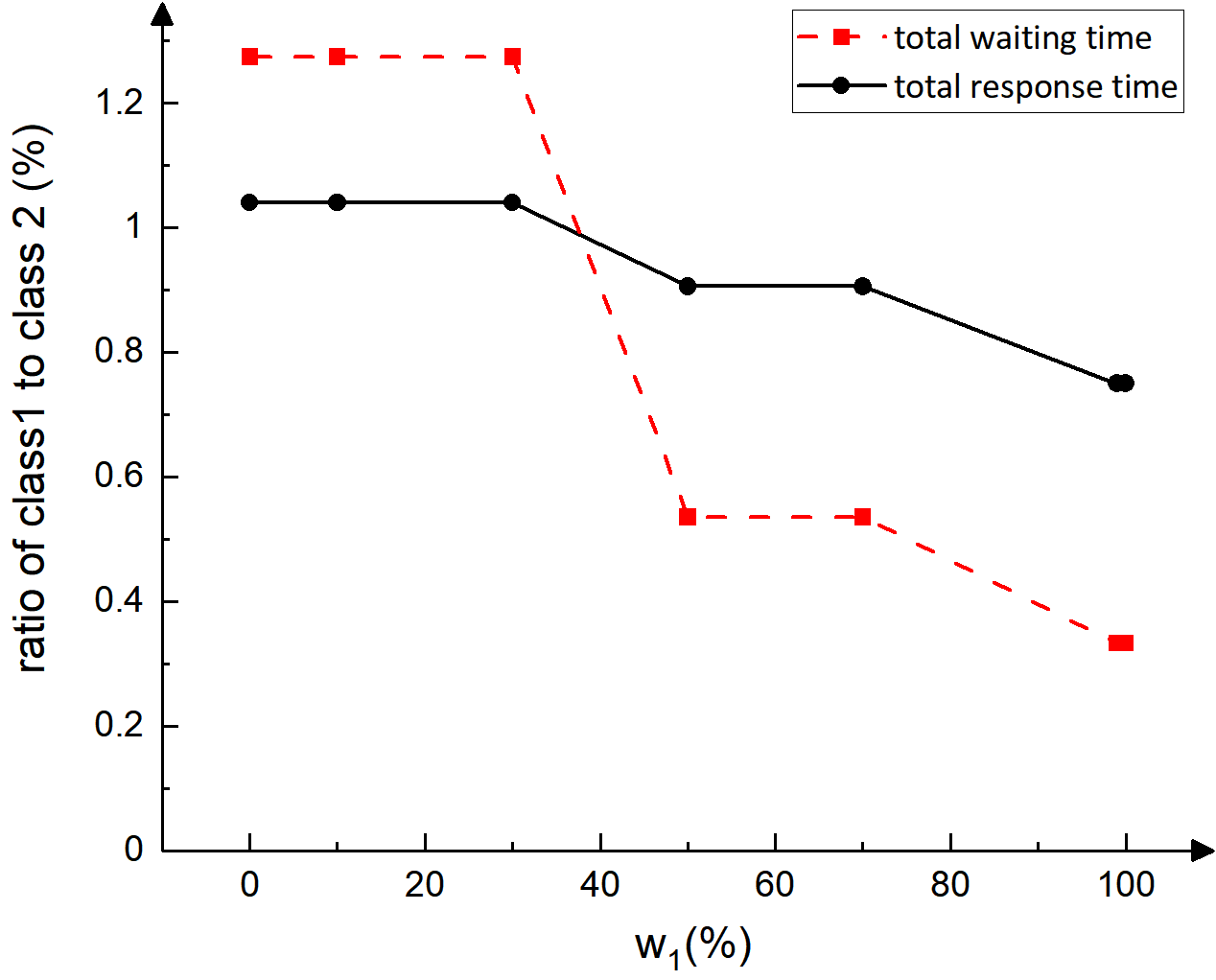}
\end{minipage}%
}%
\subfigure[$\Delta a=10$]{
\begin{minipage}[t]{0.33\linewidth}
\centering
\includegraphics[width=2in]{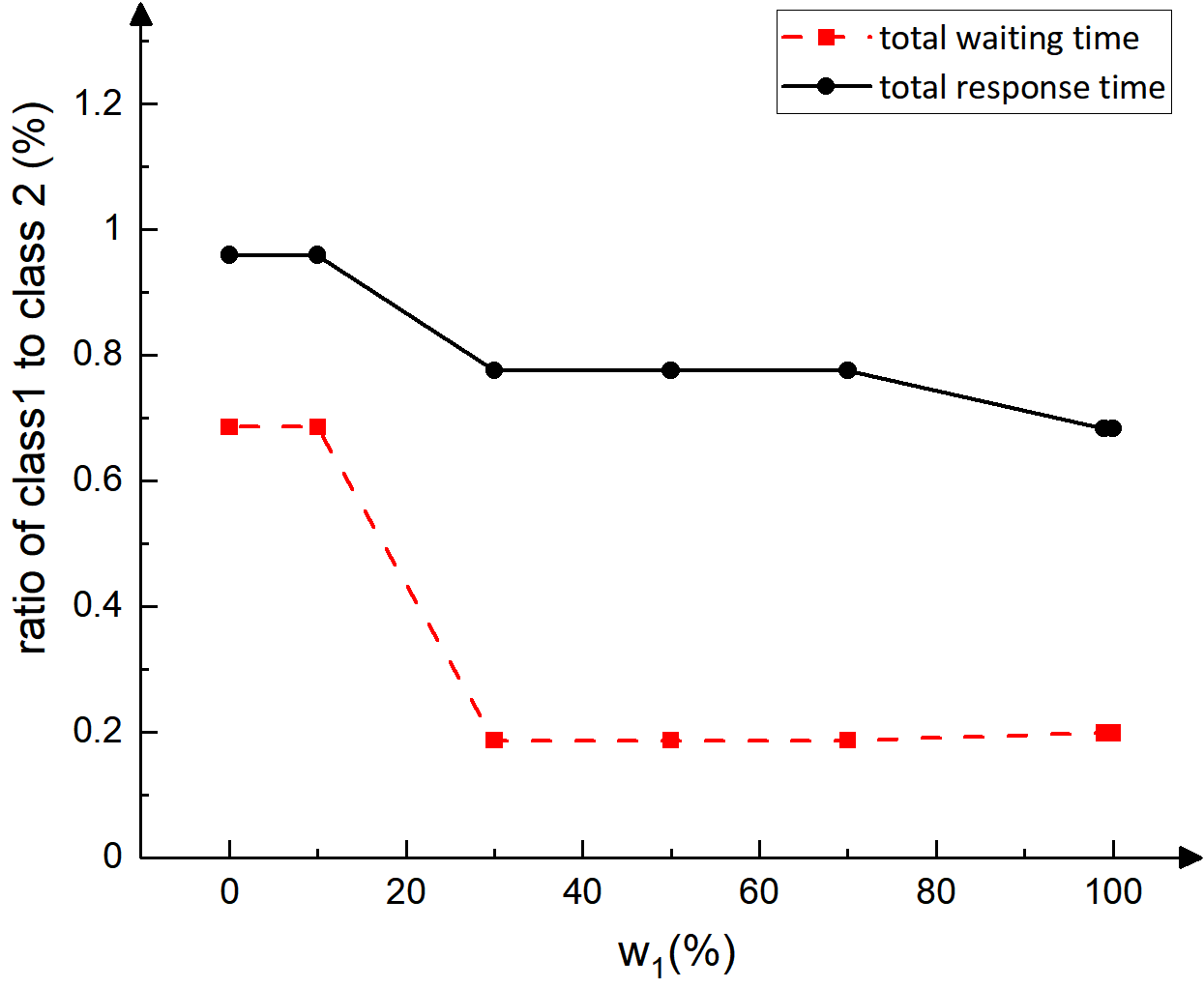}
\end{minipage}%
}%
\subfigure[$\Delta a=20$]{
\begin{minipage}[t]{0.33\linewidth}
\centering
\includegraphics[width=2in]{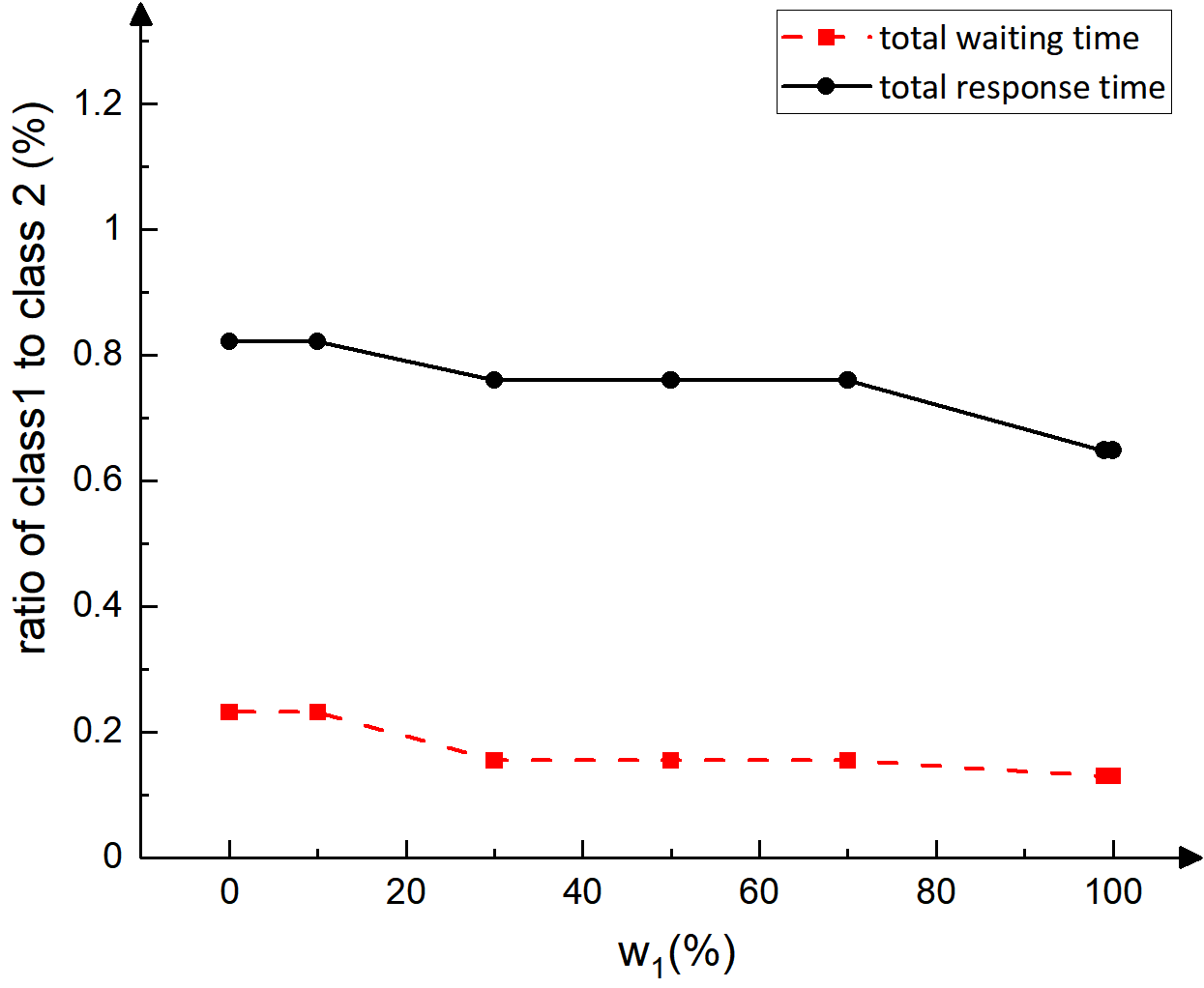}
\end{minipage}
}%

\centering
\caption{ Ratio of total response time and  total waiting time  of class 1 to class 2  demand by varying $w_1$ with $\alpha=0.1$}
\label{d21}
\end{figure}

\subsubsection{Comparison of the three priority systems}
\ 
\newline
\indent
In this section, the comparisons of   performance among  dynamic priority,  static priority, and  non-priority systems are provided. The priority class of each demand node and the objective are the same under the three  priority disciplines.   All experiments are performed with $\alpha=0.1$. 
We use DP for dynamic problem, SP for static problem,  and NP for non-priority problem for simplicity.

The box plot in Figure \ref{d3box} displays the gap of the maximum expected response and waiting time between the dynamic priority system and the static priority  or non-priority system, which takes the value
\begin{equation*}
    Gap=\frac{T_{DP}-T_{SP}(T_{NP})}{T_{SP}(T_{NP})}\times 100\%,
\end{equation*}
where $T_{DP}$ represents the  response  or waiting time in the dynamic priority system,  $T_{NP}$ and $T_{SP}$ indicate those in the non-priority and static priority systems. 
 20 instances are randomly generated under each setting. 
Compared to  NP,  DP sees lower response and waiting time for class 1 demands and higher values for class 2 demands, which is also the case when comparing SP with NP (see  section \ref{npsp}).  However, when DP is compared to  SP (see the red boxes in the figure), the trend is completely the opposite.
  This agrees with the intuition that when compared with a static priority system, dynamic priority discipline  reduces  precedence of class 1 demands, leading to a longer waiting time of class 1 and a shorter waiting time of class 2. However, the opposite effect happens when compared with a non-priority system.

\begin{figure}[!htb]
\centering 
\includegraphics[width=3.5in]{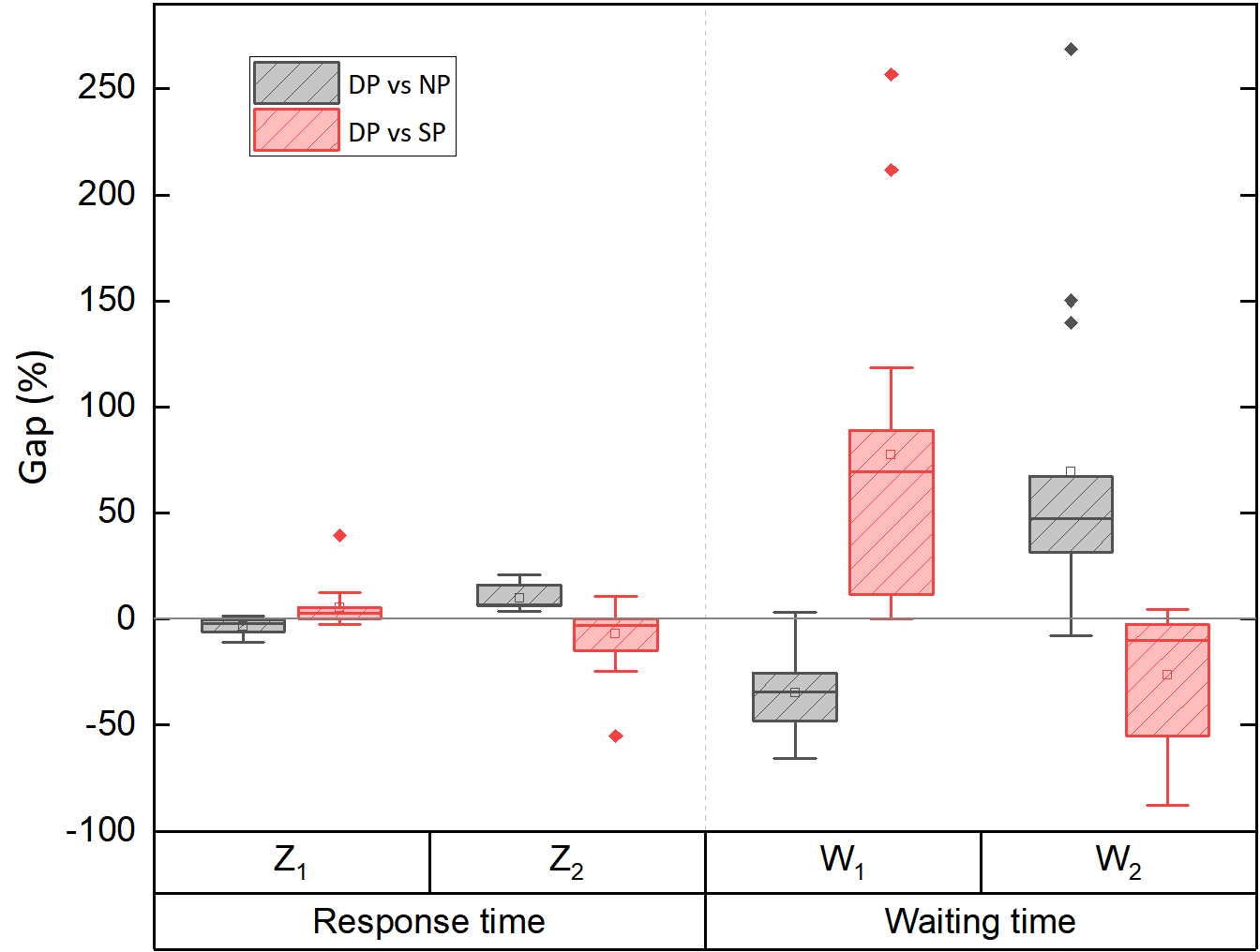} 
\caption{ Gap of maximum response and waiting time between the dynamic priority system and the static priority or non-priority system  with $\Delta a=3$ } 
\label{d3box} 
\end{figure}

Actually, the congested facility location-allocation system with static priority and without priority are two special cases of the system with dynamic priority when $\Delta a$ is set to  $+\infty$ and 0 respectively.  The results in Figure \ref{d32} support this point: the  gap  of the response time between the two classes equals 0 when $\Delta a=0$ compared with non-priority system and  gradually approaches 0 as $\Delta a$ becomes extremely large compared with the  static priority system. 

\begin{figure}[!htb]
\centering 
\includegraphics[width=3.5in]{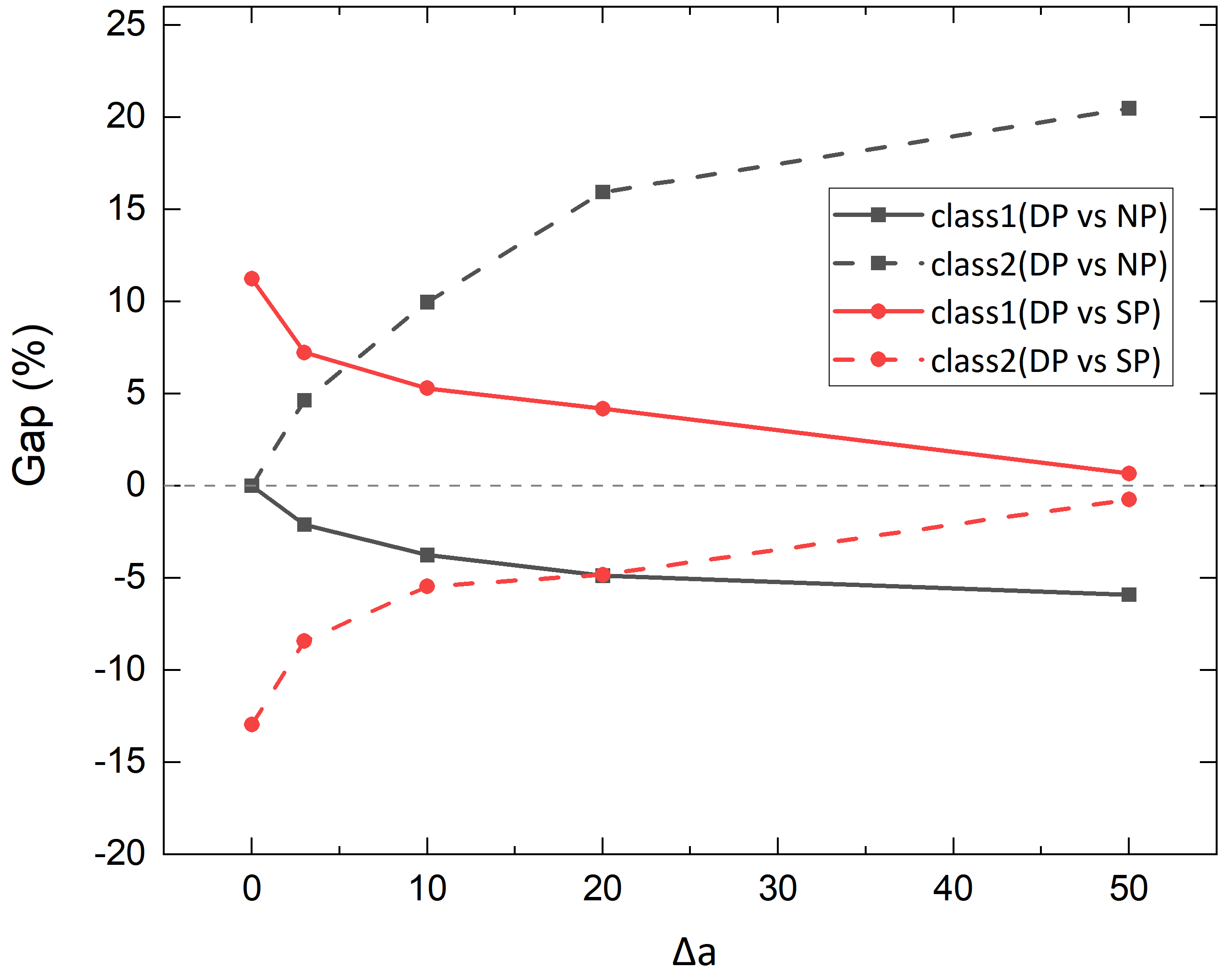} 
\caption{ Gap of maximum response  time between the dynamic priority system and the static priority or non-priority system by varying $\Delta a$ } 
\label{d32} 
\end{figure}

In addition to the overall  value of system performance, we further investigate the waiting time of each individual in our emergency delivery system. The waiting time of each demand with priority $r$ is collected through the simulation process, denoted by  $w_r$.  Tail probability is introduced, which indicates the probability of $w_r$ larger than some tail value $t$,  that is  $P(w_r>t)$. 
 Figure \ref{d34} displays the changes of the tail probability of waiting time as $t$ increases for the two priority classes in the dynamic and static systems with $\Delta a=3$, $\alpha=0.1$.  We  observe the decreasing  monotonicity in all the four groups, and the tail probability of class 1 demand from the static priority system declines faster than that from the dynamic priority system, while  the opposite is the case for class 2 demand. The strict priority discipline in the  static priority system leads to the excessive waiting time of some class 2 demands. As can be seen, about 20 percent of demands need to wait for more than 10 time units. This observation suggests that dynamic priority can effectively eliminate the extremely long waiting time of the lower priority class when adopting the priority mechanism.

\begin{figure}[!htb]
\centering 
\includegraphics[width=3.5in]{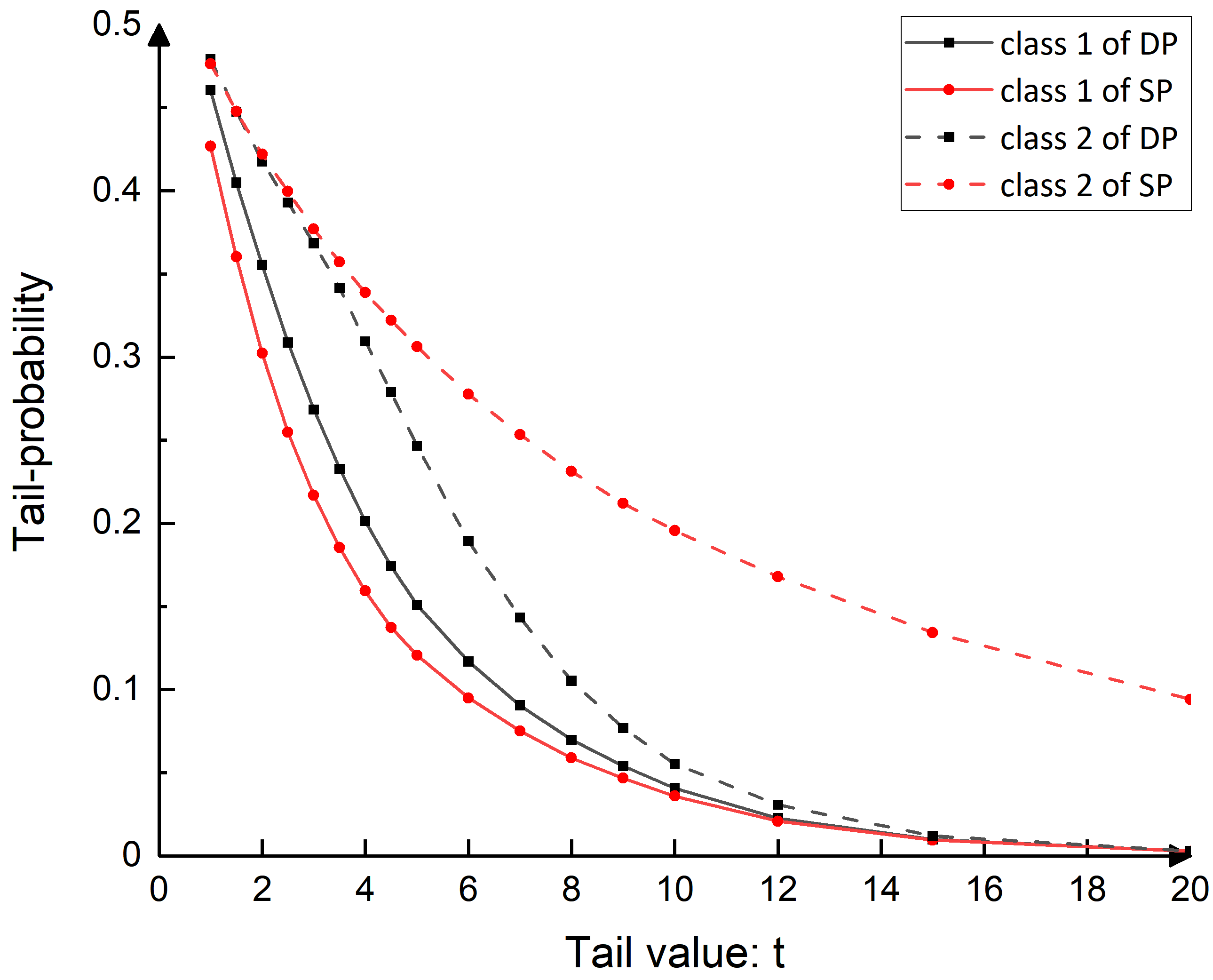} 
\caption{ Tail probability $P(W_r>t)$ } 
\label{d34} 
\end{figure}

\section{Conclusions and future work}\label{section7}
In this paper, the queuing-location models with three  priority disciplines are proposed, which are FCFS, static priority,  and dynamic priority. The optimal decisions of the facility location, demand  allocation, and deployment of drones are made simultaneously. In the model, drones are regarded as mobile servers with a general service time  related to the allocation and capacity decisions. 
 The worst expected response time among all demands is minimized in our problem concerning the equity of the emergency service. According to the nonlinear nature of the proposed models, we reformulate the original models into  mixed-integer second-order cone programs that can be solved with computational efficiency.

The optimal decisions from the mathematical models are  input to the simulation to conduct the sensitivity analysis and  evaluate the system performance.  The observation from the results suggests the unnecessity  to distribute excessive drones at a higher cost. According to the  instances  analyzed in this study, it is recommended to set  the number of   drones with $\alpha$  around 0.2.
The weight coefficient of  the priority in  the static priority system is ineffective  to adjust the importance of the priority classes, but it works in the dynamic system, especially under a  small initial priority-class gap. 
The results also show the necessity to apply  suitable priority disciplines  for different situations to improve the system performance and provide better service.
Static priority discipline is more suitable when there are  demands of  a particularly  urgent need,  since it can significantly reduce the response time of the high priority demands. 
Dynamic priority can effectively eliminate the extremely long waiting time of the lower priority class when adopting the priority mechanism. Therefore, the impact of the  dynamic priority lies between  the static priority and non-priority disciplines, and   can be adjusted by the initial priority-class gap.

When dealing with the  dynamic priority system, the upper bound of the waiting time is used, leading to the  non-optimality of the final results in several instances. Our future research will try to use the exact value of waiting time in the facility location-allocation problem with dynamic priority.
Besides,  inseparable demands are assumed in our problem, where each demand can only be assigned to exactly one facility. In the future,  we can consider separable demands and allow them to be allocated to several different facilities.  Another interesting and useful research direction is  to extend the problem to robust or stochastic programming models to address the uncertainty  factors such as the travel time or user demands.


\end{document}